\documentclass[11pt]{amsart}

\usepackage[colorlinks = true,
            linkcolor = red,
            urlcolor  = magenta,
            citecolor = black,
            anchorcolor = blue]{hyperref}
\usepackage{color}
\usepackage[shortalphabetic,initials,nobysame]{amsrefs}
\usepackage{upgreek}
\usepackage{tikz}
\usepackage{youngtab}
\usepackage[applemac]{inputenc} % for Macs
\usetikzlibrary{arrows}
\usepackage{xcolor}
\usepackage[all]{xy}
\usepackage{graphicx}
\usepackage{wrapfig}
\usepackage{yfonts}
\usepackage{fixmath}
\usepackage{graphicx}
\usepackage{amssymb}
\usepackage{epstopdf}
\usepackage{mathrsfs}
\usepackage[mathcal]{eucal}
\usepackage{bm}
\usepackage{comment}

\renewcommand{\eprint}[1]{\href{https://arxiv.org/abs/#1}{#1}}

\BibSpec{article}{%
    +{}  {\PrintAuthors}                {author}
    +{,} { \textit}                     {title}
     +{,} {}                            {note}
    +{.} { }                            {part}
    +{:} { \textit}                     {subtitle}
    +{,} { \PrintContributions}         {contribution}
    +{.} { \PrintPartials}              {partial}
    +{,} { }                            {journal}
    +{}  { \textbf}                     {volume}
    +{}  { \PrintDate}                {date}
    +{,} { \issuetext}                  {number}
    +{,} { \eprintpages}                {pages}
    +{,} { }                            {status}
    +{,} { \DOI}                   {doi}
    +{,} { \eprint}        {eprint}
      +{,} {\publisher}                {publisher}
    +{,} { \address}                   {address}
    +{}  { \parenthesize}               {language}
    +{}  { \PrintTranslation}           {translation}
    +{;} { \PrintReprint}               {reprint}
    +{.} {}                             {transition}
    +{}  {\SentenceSpace \PrintReviews} {review}
}

\newtheorem{Thm}{Theorem}[section]
\newtheorem{Lem}[Thm]{Lemma}
\newtheorem{Prop}[Thm]{Proposition}
\newtheorem{Cor}[Thm]{Corollary}
\newtheorem{Con}[Thm]{Conjecture}

\theoremstyle{definition}

\theoremstyle{remark}
\newtheorem{Rem}[Thm]{Remark}

\newtheoremstyle{named}{}{}{\itshape}{}{\bfseries}{.}{.5em}{#1 #3}
\theoremstyle{named}

\def\Z{\mathbb{Z}}

\def\g{\mathfrak{g}}
\def\Frenkel:2013uda{\mathfrak{h}}

\def\cL{\mathcal{L}}

\def\cV{\mathcal{V}}
\def\cW{\mathcal{W}}

\def\aa{\mathbold{a}}
\def\degs{\mathbold{d}}
\def\uu{\mathbold{u}}
\def\zz{\mathbold{z}}
\def\xx{\mathbold{x}}

\def\bo{\textbf{o}}

\def\=>{\Longrightarrow}

\def\to{\longrightarrow}

\def\o+{\oplus}
\def\bo+{\bigoplus}

\def\<{\langle}
\def\>{\rangle}
\def\({\left(}
\def\){\right)}

\def\^{\wedge}
\def\+{\dagger}

\def\dd[#1,#2]{\frac{d#1}{d#2}}
\def\del[#1,#2]{\frac{\partial #1}{\partial #2}}
\def\over[#1]{\overline{#1}}
\def\vec[#1]{\overrightarrow{#1}}

\makeatletter
\def\mr@ignsp#1 {\ifx\:#1\@empty\else #1\expandafter\mr@ignsp\fi}%
\newcommand{\multiref}[1]{\begingroup%\let\protect\string%
\xdef\mr@no@sparg{\expandafter\mr@ignsp#1 \: }%
\def\mr@comma{}%
\@for\mr@refs:=\mr@no@sparg\do{\mr@comma\def\mr@comma{,}\ref{\mr@refs}}%
\endgroup}
\makeatother

\newcommand{\hypref}[2]{\ifx\href\asklFrenkel:2013udaas #2\else\href{#1}{#2}\fi}

\tikzset{->-/.style={decoration={
  markings,
  mark=at position .5 with {\arrow{latex}}},postaction={decorate}}}
\tikzset{
    >=latex
    }

\newcommand{\nc}{\newcommand}
\nc{\on}{\operatorname}
\nc{\la}{\lambda}
\nc{\wh}{\widehat}
\nc{\ghat}{\wh\g}
\nc{\mb}{\mathbf}

\begin{document}
\title{Quantum K-theory of Quiver Varieties at Roots of Unity}

\author[P. Koroteev]{Peter Koroteev}
\address{
Department of Mathematics,
University of California,
Berkeley, CA 94720, USA
and
Beijing Institute for Mathematical Sciences and Applications,
Beijing Huairou District, 101408, China
}

\author[A. Smirnov]{Andrey Smirnov}
\address{
Department of Mathematics,
University of North Carolina at Chapel Hill, 
Chapel Hill, NC 27599, USA
}

\date{\today}

\numberwithin{equation}{section}

\begin{abstract}
Let $\Psi(\zz,\aa,q)$ be a fundamental solution matrix of the quantum difference equation of a Nakajima variety $X$. In this work, we prove that the operator
$$
\Psi(\zz,\aa,q) \Psi(\zz^p,\aa^p,q^{p^2})^{-1}
$$
has no poles at the primitive complex $p$-th roots of unity $q=\zeta_p$. As a byproduct, we show that the iterated product of the operators ${\bf M}_{\cL}(\zz,\aa,q
)$ from the $q$-difference equation on $X$: 
$$
{\bf M}_{\cL} (\zz q^{(p-1)\cL},\aa,q)   \cdots {\bf M}_{\cL} (\zz q^{\cL},\aa,q)  {\bf M}_{\cL} (\zz ,\aa,q)
$$
evaluated at $q=\zeta_p$ has the same eigenvalues as ${\bf M}_{\cL} (\zz^p,\aa^p,q^p)$.

Upon a reduction of the quantum difference equation of $X$ to the quantum differential equation over the field of finite characteristic, the above iterated product transforms into a Gro\-then\-diek-Katz $p$-curvature of the corresponding quantum connection whe\-re\-as ${\bf M}_{\cL} (\zz^p,\aa^p,q^p)$ becomes a certain Frobenius twist of that connection. In this way, we give an explicit description of the spectrum of the $p$-curvature of quantum connection for Nakajima varieties. 
\end{abstract}

\maketitle

%\setcounter{tocdepth}{1}
%\tableofcontents

\section{Introduction}
\subsection{The Quantum Difference Equation}
Enumerative algebraic geometry (quan\-tum K-the\-ory) of Nakajima varieties is governed by quantum difference equations (QDE) \cite{Okounkov:2022aa} which have the following form
\begin{align} \label{QDEintro}
\Psi(\zz q^{\cL},\aa,q) \cL  = {\bf M}_{\cL} (\zz,\aa,q) \Psi(\zz,\aa,q), \ \ \ \cL \in \mathcal{Z}\,,
\end{align}
where $\mathcal{Z} \cong \mathbb{Z}^l$ is the lattice of tautological line bundles on a Nakajima variety $X$ and $l$ denotes the number of vertices in the corresponding quiver. The variables $\zz=(z_1,\dots,z_l)$ and $\aa=(a_1,\dots, a_m)$ denote the K\"ahler and the equivariant parameters, respectively. The shift of the K\"ahler variables is of the form
$$
\zz q^{\cL} =(z_1 q^{c_1},\dots, z_{l} q^{c_{l}} ),
$$
where $c_i \in \mathbb{Z}$ are integers determined by the expansion 
\begin{align} \label{linbuns}
\cL = L_1^{c_1} \otimes \dots \otimes L_l^{c_l},
\end{align}
in the basis of the lattice $\mathcal{Z}$ given by the tautological line bundles $L_i$ 

Let $\Psi(\zz,\aa,q)$ be the fundamental solution matrix of (\ref{QDEintro}) given by a power series in $\zz$ and uniquely determined by the normalization 
\begin{align} \label{powerser}
\Psi(\zz,\aa,q) =\sum\limits_{d\in C} \Psi_d(\aa,q) \zz^{d}   \in K_{T}(X)[[\zz]],
\end{align}
where $C \subset \mathcal{Z}$ is a cone  determined  by a choice of stability condition for $X$.

The matrix $\Psi(\zz,\aa,q)$ provides the {\it capping operator} -- the fundamental object in enumerative geometry which can be defined as the partition function of quasimaps with relative and non-singular boundary conditions, see Section 7.4 in \cite{Okounkov:2015aa} for definitions.

Difference connection (\ref{QDEintro}) is flat, i.e., for any two line bundles $\cL_1,\cL_2 \in \mathcal{Z}$ we have
\begin{align} \label{Mconsis}
{\bf M}_{\cL_1 \cL_2} (\zz,\aa,q) =  {\bf M}_{\cL_2} (\zz q^{\cL_1},\aa,q) {\bf M}_{\cL_1} (\zz,\aa,q) = 
{\bf M}_{\cL_1} (\zz q^{\cL_2},\aa,q) {\bf M}_{\cL_2} (\zz,\aa,q). 
\end{align}
Explicit formulae for ${\bf M}_{\cL} (\zz,\aa,q)$ in terms of representation theory of quantum groups were obtained in \cite{Okounkov:2022aa}. An alternative description of these operators in terms of the elliptic stable envelope classes was also obtained in \cites{2020arXiv200407862K, Kononov:2020cux}. In any chosen basis of the equivariant $K$-theory $K_{T}(X)$ the operators ${\bf M}_{\cL} (\zz,\aa,q)$ are represented by matrices with coefficients given by rational functions in $\mathbb{Q}(\zz,\aa,q)$. 

%It is also known that ${\bf M}_{\cL}(\zz,\aa,q)$ are nonsingular at any $q\in \mathbb{C}$.  

\subsection{Quantum K-theory} 
It follows from the definition of the shift operators ${\bf M}_{\cL} (\zz,\aa,q)$, namely, from the properness of the  relative quasimap moduli space, that they do not have poles in $q$, see Section 8.1 in \cite{Okounkov:2015aa}. In particular, these operators have well-defined specializations at $q=1$. 

Let us consider the following operators 
\begin{align} \label{qprod}
\mathcal{M}_{\cL}(\zz,\aa) =  \left.{\bf M}_{\cL} (\zz,\aa,q)\right|_{q=1}.
\end{align}
From (\ref{Mconsis}) we see that 
$$
\mathcal{M}_{\cL_1 \cL_2}(\zz,\aa) = \mathcal{M}_{\cL_1}(\zz,\aa) \mathcal{M}_{\cL_2}(\zz,\aa) =\mathcal{M}_{\cL_2}(\zz,\aa)  \mathcal{M}_{\cL_1}(\zz,\aa)
$$
i.e. these operators commute
\begin{align} \label{commutQ}
[\mathcal{M}_{\cL_1}(\zz,\aa),\mathcal{M}_{\cL_2}(\zz,\aa)] =0, \ \ \ \forall \cL_1, \cL_2 \in \mathcal{Z}\,.
\end{align}
In  \cites{Pushkar:2016qvw,Koroteev:2017aa} we showed that  $\mathcal{M}_{\cL}(\zz,\aa)$ are the operators of quantum multiplication by $\cL$ in the  equivariant quantum K-theory ring $QK_{T}(X)$. This ring is commutative which agrees with (\ref{commutQ}).

An interesting problem is to describe  the joint set of eigenvalues and eigenvectors of the operators $\mathcal{M}_{\cL}(\zz,\aa)$. It is conjectured that the joint spectrum of $\mathcal{M}_{\cL}(\zz,\aa)$, $\cL \in \mathcal{Z}$ is simple\footnote{This conjecture is known for type $A$-flag varieties $X=T^{*}(G/P)$, since in this case the classical multiplication by generic $\cL$ has simple spectrum. It is also known for $X=Hilb^n(\mathbb{C}^2)$, via cohomological limit \cite{2004math.....11210O}. We expect that it can also be proved for affine type $A$ quiver varieties by carefull analysis of Bethe roots, e.g., \cite{2024arXiv240909508S}.  }, which implies that they generate the quantum K-theory ring $QK_{T}(X)$. 

%It is expected that this spectrum is simple meaning that the common eigenspaces of all $\mathcal{M}_{\cL}(\zz,\aa)$ are one-dimensional. Among other things the simplicity implies that $\mathcal{M}_{\cL}(\zz,\aa)$ generate the quantum K-theory ring $QK_{T}(X)$ -- the operators of quantum multiplication by any K-theory class can be expressed as a polynomial in  $\mathcal{M}_{\cL}(\zz,\aa)$, $\cL\in \mathrm{Pic}(X)$ with certain rational coefficients in $\zz$ and $\aa$. 

\subsection{Eigenvalues of $\mathcal{M}_{\cL}(\zz,\aa)$ and Bethe Ansatz}
The above-mentioned eigenvalue problem also arises naturally in the theory of quantum integrable spin chains \cites{Pushkar:2016qvw,Koroteev:2017aa}. For any quiver variety $X$ there is a quantum group $\mathcal{U}_{\hbar}(\widehat{\mathfrak{g}}_X)$ which acts on its equivariant K-theory $K_{T}(X)$, see Section 3 of \cite{Okounkov:2022aa} for the construction. This action identifies $K_{T}(X)$ with the quantum Hilbert space of a certain XXZ-type spin chain. In this setting, the algebra of commuting Hamiltonians of the spin chain is identified with the algebra generated by operators of quantum multiplication by the K-theory classes, i.e., with the commutative algebra $QK_T(X)$. In particular, the operators $\mathcal{M}_{\cL}(\zz,\aa)$ represent certain Hamiltonians of the corresponding XXZ spin chain. Namely, 
the operators of quantum multiplication by line bundles 
$\mathcal{M}_{\cL}(\zz,\aa)$ appear as the ``top'' coefficients of the Baxter $Q$-operators of the spin chain \cite{Pushkar:2016qvw}. Describing the eigenvalues and eigenvectors of these Hamiltonians is a classical problem in quantum mechanics. 

\vskip.1in

The algebraic Bethe Ansatz \cite{Faddeev:1996iy} is a method used in the theory of integrable models to diagonalize spin chain Hamiltonians.  Let $\mathcal{V}_{i}$, $i=1,\dots, l$ be the set of tautological bundles over a Nakajima variety $X$. Let $\xx=\{x_{i,j}\}$ denote the collection of the Grothendieck roots of these vector bundles, so that in K-theory we have:
\begin{align} \label{tautinroots}
\mathcal{V}_{i} = x_{i,1}+\dots +x_{i,r_i}, \ \ r_i=\textrm{rk}{\mathcal{V}_{i}}.
\end{align}
The tautological line bundles are given by $L_i= \det \mathcal{V}_{i} = x_{i,1}\cdot \dots \cdot x_{i,r_i}$. 
By (\ref{linbuns}) every line bundle $\cL$ can be represented as a certain product of the Grothendieck roots 
\begin{equation}\label{Lbundle}
\cL =  \prod_{i=1}^{l} \left(\prod_{j=1}^{r_i} x_{i,j}\right)^{c_i} \,.
\end{equation}
In a nutshell, the algebraic Bethe Ansatz asserts \cite{Aganagic:2017be} that the eigenvalues of $\mathcal{M}_{\cL}(\zz,\aa)$ are given by {the same product}
\begin{align} \label{eigenvalbeth}
\lambda(\aa,\zz) =\prod_{i=1}^{l} \left(\prod_{j=1}^{r_i} x_{i,j}\right)^{c_i} \,,
\end{align}
where $x_{i,j}$  are now certain functions of $\zz$ and $\aa$ determined as the roots of the algebraic equations, known as the {Bethe equations}:
\begin{align} \label{betheeqintro}
\mathfrak{B}(\xx,\zz,\aa) =0.
\end{align}
Equations (\ref{betheeqintro}) are constructed explicitly from the underlying quiver \cite{Aganagic:2017be}. We recall this construction in Section \ref{betheeqseq}. In essence, each solution to the Bethe equations (\ref{betheeqintro}) provides $x_{i,j}$ as specific functions of the parameters $\zz$ and $\aa$. Substituting those functions into (\ref{eigenvalbeth}) gives an eigenvalue of $\mathcal{M}_{\cL}(\zz,\aa)$.

\subsection{The Case $q^p=1$} 
Let $p\in \mathbb{N}$ and let $\zeta_p \in \mathbb{C}$ be a primitive $p$-th root of unity\footnote{At this point $p$ does not have to be prime, that would be required later in Section \ref{Sec:pCurvFrob}.}. 
For a line bundle $\cL$ we consider the operator ${\bf M}_{\cL^p} (\zz,\aa,q)$. By iterating (\ref{Mconsis}) $p$-times we have:
\begin{equation} \label{eq:QpoperFull}
{\bf M}_{\cL^p}(\zz,\aa,q) = {\bf M}_{\cL} (\zz q^{(p-1)\cL},\aa,q)  {\bf M}_{\cL} (\zz q^{(p-2) \cL},\aa,q)  \cdots {\bf M}_{\cL} (\zz q^{ \cL},\aa,q)  {\bf M}_{\cL} (\zz,\aa,q).
\end{equation}
We denote its value at $q=\zeta_p$ by:
\begin{align} \label{Qpoper}
\mathcal{M}_{\cL,\zeta_p}(\zz,\aa) = {\bf M}_{\cL^p}(\zz,\aa,\zeta_p)\,.
\end{align}
It is evident from (\ref{Mconsis}) that
$$
\mathcal{M}_{\cL_1 \cL_2,\,\zeta_p}(\zz,\aa) = \mathcal{M}_{\cL_1,\,\zeta_p}(\zz,\aa)  \mathcal{M}_{\cL_2 ,\,\zeta_p}(\zz,\aa) =\mathcal{M}_{\cL_2 ,\,\zeta_p}(\zz,\aa)  \mathcal{M}_{\cL_1,\,\zeta_p}(\zz,\aa).
$$
In particular, these operators commute 
$$
[\mathcal{M}_{\cL_1 \,\zeta_p}(\zz,\aa) ,\mathcal{M}_{\cL_2 \,\zeta_p}(\zz,\aa)] =0, \ \ \ \forall \cL_1, \cL_2 \in \mathcal{Z}.
$$
It is, therefore, natural to study the joint set of eigenvalues for these operators. Surprisingly, this problem has not been considered yet.

In this paper we prove that the eigenvalues of $\mathcal{M}_{\cL,\zeta_p}(\zz,\aa)$ can be obtained from the eigenvalues of $\mathcal{M}_{\cL}(\zz,\aa)$ as follows.

\begin{Thm} \label{specthmintro}
Let $\{\lambda_1(\zz,\aa),\lambda_2(\zz,\aa), \dots\}$ be the set of the eigenvalues of $\mathcal{M}_{\cL}(\zz,\aa)$ then the eigenvalues of $\mathcal{M}_{\cL \,\zeta_p}(\zz,\aa)$ are given by the set
$\{\lambda_1(\zz^p,\aa^p),\lambda_2(\zz^p,\aa^p),\dots\}$
where 
$\lambda_i(\zz^p,\aa^p)$
is the eigenvalue $\lambda_i(z,a)$ in which all K\"ahler variables $\zz=(z_1,\dots, z_l)$ and equivariant variables $\aa=(a_1,\dots, a_m)$ are substituted by $\zz^p=(z^p_1,\dots, z^p_l)$  and $\aa^p=(a^p_1,\dots, a^p_m)$ respectively. 
\end{Thm}

%Note that by above theorem, the eigenvalues do not depend on the choice of a primitive root $\zeta_p$, thus  all the operators $\mathcal{M}_{\cL,\zeta_p}(\zz,\aa)$ have the same spectrum and therefore are related by a conjugation transformation. 

As we explain in the previous subsection, the eigenvalues of $\mathcal{M}_{\cL}(\zz,\aa)$ can be determined from the Bethe equations. In this way, the spectrum of $\mathcal{M}_{\cL ,\zeta_p}(\zz,\aa)$ is also controlled by these equations.

\subsection{The Intertwiner} 
Theorem \ref{specthmintro} is a corollary of  the following pole cancellation property of the fundamental solution matrix $\Psi(\zz,\aa,q)$. The coefficients of the power series expansion (\ref{powerser}) have poles in $q$ located at the roots of unity. 
In particular, $\Psi(\zz,\aa,q)$ is singular at $q = \zeta_p$. 

Let $\Psi(\zz^p,\aa^p,q^{p^2})$ be the fundamental solution matrix in which all K\"ahler and equivariant parameters are raised to power $p$ while the variable $q$ is raised to the power $p^2$. We prove the following pole-cancellation property.
\begin{Thm}
The operator
 $
\Psi(\zz,\aa,q) \Psi(\zz^p,\aa^p,q^{p^2})^{-1}
 $  
 has no poles in $q$ located at  primitive complex $p$-th roots of unity. 
 %In particular, it is well defined at $q=\zeta_p$.  
\end{Thm}
\noindent
Let us define
\begin{equation}\label{eq:Fdef}
\mathsf{F}(\zz,\aa,\zeta_p) = \Psi(\zz,\aa,q) \Psi(\zz^p,\aa^p,q^{p^2})^{-1}\Big\vert_{q=\zeta_p}
\end{equation}
It can be shown from the difference equation (\ref{QDEintro}) that
\begin{align} \label{conjintro}
\mathsf{F}(\zz,\aa,\zeta_p) \mathcal{M}_{\cL}(\zz^p,\aa^p)  \mathsf{F}(\zz,\aa,\zeta_p)^{-1} = \mathcal{M}_{\cL,\zeta_p}(\zz,\aa)\,,
\end{align}
where $\mathcal{M}_{\cL}(\zz^p,\aa^p)$ denotes the operator of quantum multiplication (\ref{qprod}) in which all variables are raised to power $p$. The isospectrality Theorem \ref{specthmintro} thus follows as an obvious corollary to (\ref{conjintro}). 

The operator \eqref{eq:Fdef} is a complex field analog of the Frobenius intertwiner \cite{Smirnov:2024yqs} which is defined over~$\mathbb{Q}_p$. We note, however, a difference -- the intertwiner constructed in \cite{Smirnov:2024yqs} does not have poles at $p$-adic roots of unity of order $p^s$ for any $s$, while \eqref{eq:Fdef} is only defined at the roots of the order $p$.

\subsection{$p$-Curvature and Frobenius} 
The concept of $p$-curvature originated in Gro\-then\-dieck's unpublished work from the 1960s and was subsequently developed further by Katz \cites{Katz2,Katz1}. The $p$-curvature plays an important role in the theory of ordinary differential equations (ODEs) as well as holonomic PDEs, establishing a connection between the existence of algebraic fundamental solutions and their behavior under reduction modulo a prime~$p$. Specifically, if an algebraic solutions exist, then for almost all primes the reduction of the ODE modulo $p$
exhibits zero $p$-curvature. The converse, however, remains an open question and is known as the Grothendieck-Katz conjecture.

Recently, Jae Hee Lee gave an enumerative interpretation of the $p$-curvature operators \cite{Lee:2023aa}. In \cites{fukaya1996morse, wilkins2020} a new class of operators was defined in the study of quantum cohomology modulo a prime $p$, which are known as {\it quantum Steenrod operations}. In his work, Jae Hee Lee showed that the quantum Steenrod operations coincide with the $p$-curvature of quantum connection for a large class of symplectic resolutions, including Nakajima varieties with isolated torus fixed points under the assumption of simplicity of the quantum multiplication. In \cite{2025arXiv250323590B} the agreement between the $p$-curvature and quantum Steenrod operations has also been shown for hypertoric varieties. Most recently, in \cite{2025arXiv251009335B} Bai and Lee discussed quantum Adams operations which are the K-theoretic generalizations of the quantum Steenrod operations, and are studied in this paper in the setting of quasimaps.

In Section \ref{Sec:pCurvFrob} we show that our operators $\mathcal{M}_{\cL,\zeta_p}(\zz,\aa)$ (\ref{Qpoper}) provide a proper $K$-theoretic generalization of the $p$-curvature once all its parameters are specialized to their $p$-adic values. Namely, we consider an extension of the $p$-adic field $\mathbb{Q}_p(\pi)$ where $\pi$ solves the equation $\pi^{p-1}=-p$. The ideal $(\pi)$ in the ring of integers $\mathbb{Z}_p[\pi] \subset \mathbb{Q}_p(\pi)$ of this field is maximal with the quotient field $\mathbb{Z}_p[\pi]/(\pi) = \mathbb{F}_p$. Using this property, we analyze operator $\mathcal{M}_{\cL,\zeta_p}(\zz,\aa)$ near $\zeta_p \in \mathbb{Q}_p$ given by a primitive $p$-th root of unity, and then reduce it modulo $(\pi)$ to the finite field ~$\mathbb{F}_p$. We show that under this reduction to $\mathbb{F}_p$ the operator $\mathcal{M}_{\cL,\zeta_p}(\zz,\aa)$ specializes precisely to the $p$-curvature of the quantum connection on $X$. Our main isospectrality Theorem \ref{specthmintro} then reduces to a result describing the eigenvalues of the $p$-curvature:

\begin{Thm} \label{intcurthm}
The $p$-curvature $C_p(\nabla_i)$ of the quantum connection for a Nakajima variety
$$
\nabla_i =   z_i \frac{\partial}{\partial z_i} - s\, C_i(\zz,\uu), \ \ \ i=1,\dots, l,
$$
and the Frobenius twist of the quantum multiplication by the divisor $(s^p-s)C_i(\zz,\aa)^{(1)}$ have the same spectrum over~$\mathbb{F}_p$.  
\end{Thm}
We refer to Section \ref{Sec:pCurvFrob} and Theorems \ref{Th:Isopectral1} and \ref{Th:EVpencil} for notations and details. 

\vskip.1in
We note also that this theorem was recently proven by Etingof and Varchenko \cites{EV2024,Etingof:2024ab}, using a very different approach -- they introduced and studied a large family of differential operators, called {\it periodic pencils}, which include among many other examples, the quantum connections of Nakajima varieties. Theorem \ref{intcurthm} is deduced in \cite{EV2024} from various strong properties of the periodic pencils and semiclassical analysis.

\subsection{Example $X=T^{*}\mathbb{P}^1$\label{examplesec}} 
It might be instructive to illustrate the statement of isospectrality Theorem \ref{specthmintro} in a simple example. Consider a Nakajima variety given by the cotangent bundle over projective line $X=T^{*} \mathbb{P}^1$. Let $T\cong (\mathbb{C}^{\times})^{3}$ be a torus with coordinates $\aa=(a_1,a_2,\hbar)$. We consider the action of $T$ on $X$ induced by the natural action on $\mathbb{C}^2$ given by
$
(x,y) \mapsto (x a_1, y a_2).
$
In addition, $T$ acts on $X$ by dilating the cotangent direction by $\hbar^{-2}$, i.e. $\hbar^{2}$ is the $T$-character of the canonical symplectic form on $X$\footnote{In this example our notation for $\hbar$ is the same as in Section 6 of \cite{Okounkov:2022aa} and which we follow here}. 

The lattice of tautological bundles in this case is $\mathcal{Z} \cong \mathbb{Z}$ generated by the tautological line bundle $\cL= \mathcal{O}(1)$. The equivariant $K$-theory is isomorphic to the following ring:
$$
K_{T}(X) \cong \mathbb{C}[\cL,a_1,a_2,\hbar]/I_0
$$
where $I_0$ denotes the ideal generated by a single relation 
$$
(\cL-a_1) (\cL-a_2) =0.
$$
The quantum K-theory ring of $X$ is a deformation of this ring:
$$
QK_{T}(X)= \mathbb{C}[\cL,a_1,a_2,\hbar][[z]]/I_z
$$
where $I_z$ is the ideal generated by the relation \cite{Okounkov:2022aa,Pushkar:2016qvw}
\begin{align} \label{qringrel}
(\cL-a_1) (\cL-a_2) = z \hbar^{2} (\cL- a_1 \hbar^{-2}) (\cL- a_2 \hbar^{-2}).
\end{align}
Specializing $QK_{T}(X)$  at $z=0$ gives back the classical K-theory $K_{T}(X)$. 

The quantum difference equation for $X$ was considered in details in Section 6 of \cite{Okounkov:2022aa}, in particular, the explicit expression for the operator $\mathbf{M}_{\cL}(z)$ is given in Section 6.3.9. In the stable basis of $K_{T}(X)$ this operator has the form:
$$
\mathbf{M}_{\cL}(\zz,\aa,q)= \left[ \begin {array}{cc} {\dfrac {a_{{1}} \left( zq-1 \right) }{{\hbar}^{2
}zq-1}}&{\dfrac {a_{{2}}z q( \hbar^{-1}-{\hbar})}{{\hbar}^{2}z q -1}  }
\\ \noalign{\medskip}{\dfrac {a_{{1}} ( {\hbar}^{-1}-\hbar)}{{\hbar}^{2}z q -1}  }&{\dfrac {a_{{2}} \left( z q-1 \right) }{{\hbar}^{2}z q-1}}
\end {array} \right]. 
$$
At $q=1$ we thus obtain the operator of quantum multiplication by $\cL$  in the stable basis of the equivariant K-theory:
\begin{equation}\label{eq:Mzrescaled}
    \mathcal{M}_{\cL}(z,\aa)= \left[ \begin {array}{cc} {\dfrac {a_{{1}} \left( z-1 \right) }{{\hbar}^{2
}z-1}}&{\dfrac {a_{{2}}z( \hbar^{-1}-{\hbar})}{{\hbar}^{2}z-1}  }
\\ \noalign{\medskip}{\dfrac {a_{{1}} ( {\hbar}^{-1}-\hbar)}{{\hbar}^{2}z-1}  }&{\dfrac {a_{{2}} \left( z-1 \right) }{{\hbar}^{2}z-1}}
\end {array} \right] 
\end{equation}
It is straightforward to check that this matrix satisfies quadratic relation (\ref{qringrel}), i.e.:
\begin{align} \label{Qchaareq}
(\mathcal{M}_{\cL}(z,\aa) -a_1) (\mathcal{M}_{\cL}(z,\aa) -a_2) - z \hbar^{2}  (\mathcal{M}_{\cL}(z,\aa) -a_1 \hbar^{-2}) (\mathcal{M}_{\cL}(z,\aa) -a_2 \hbar^{-2}) =0
\end{align}

Next, let $p\in \mathbb{N}$ and $\zeta_p$ be a primitive complex root of unity of order~$p$. Let us consider the operator (\ref{Qpoper}), which in this case equals:
$$
\mathcal{M}_{\cL,\zeta_p}(z,\aa) = \mathcal{M}_{\cL}(z \zeta^{p-1}_p,\aa) \cdots \mathcal{M}_{\cL}(z \zeta_p,\aa) \mathcal{M}_{\cL}(z,\aa)
$$
%Since in this case $\mathbf{M}_{\cL}(z)$ trivially depends on $q$ we have
%$$
%\mathcal{M}_{\cL,\zeta_p}(z,a)=  \mathcal{M}_{\cL}(z) \mathcal{M}_{\cL}(z \zeta_p) \mathcal{M}_{\cL}(z\zeta_p^2)\cdots \mathcal{M}_{\cL}(z\zeta_p^{p-1})
%$$
Using a computer one verifies that for any choice of a natural number $p$ and a primitive root of unity $\zeta_p$ this matrix satisfies the following relation
$$
(\mathcal{M}_{\cL,\zeta_p}(z,\aa) -a_1^p) (\mathcal{M}_{\cL,\zeta_p}(z,\aa) -a^p_2) = z^p \hbar^{2p}  (\mathcal{M}_{\cL,\zeta_p}(z,\aa) -a^p_1 \hbar^{-2p}) (\mathcal{M}_{\cL,\zeta_p}(z,\aa) -a^p_2 \hbar^{-2 p }).
$$
Note that this relation is obtained from (\ref{qringrel}) by raising all parameters to their $p$-th powers. Since the left side of (\ref{Qchaareq}) is nothing but the characteristic polynomial for the matrix $\mathcal{M}_{\cL}(z,\aa)$,
this implies that the eigenvalues of $\mathcal{M}_{\cL,\zeta_p}(z,\aa)$ can be obtained from the eigenvalues of $\mathcal{M}_{\cL}(z,\aa)$ via the following substitution
\begin{equation}\label{eq:FrobMap}
    z\mapsto z^p,\qquad a_1\mapsto a_1^p,\qquad a_2\mapsto a_2^p,\qquad\hbar \mapsto \hbar^p\,.
\end{equation}
This illustrates the statement of Theorem \ref{specthmintro} in this simple example.

Finally, we also note that the relation in the quantum K-theory ring (\ref{qringrel}), is nothing but the Bethe equation for $X$. Solving this quadratic equation for $\cL$ gives two eigenvalues of matrix (\ref{eq:Mzrescaled}). This illustrates how the Bethe Ansatz works in this case.

\subsection{Acknowledgments} The authors thank Pavel Etingof, Andrei Okounkov, and Aleksander Varchenko for fruitful discussions. We also thank Jae Hee Lee and Shaoyun Bai for sharing exiting ideas about constructing quantum K-theory and quantum Adams operations modulo $p$. Work of A.\,Smirnov is partially supported by NSF grants DMS-2054527,  DMS-2401380 and by the Simons Foundation grant ``Travel Support for Mathematicians''.

\section{Asymptotics of Vertex Functions \label{section2}}
The quantum difference equation of a Nakajima variety $X$ has a natural basis of solutions given by the $K$-theory components of the vertex function \cite{Aganagic:2017smx}. These functions can be viewed as generalizations of the classical $q$-hypergeometric series -- the components of the vertex function of the simplest Nakajima variety $X=T^{*}\mathbb{P}^n$ are exactly the $q$-hypergeometric series, see \cite{2020arXiv200805516D,2024LMaPh.114...23D} for explicit examples. 

Similarly to the $q$-hypergeometric functions, the vertex functions have natural integral representations of Mellin-Barnes type, see Section 3 of \cite{Aganagic:2017smx}. The integral representations can be used for computing the asymptotics of the vertex functions as $q\to 1$, or, more generally, as $q\to \zeta$, where $\zeta$ is a root of unity, using the method of steepest descent. 

The saddle point equations for the steepest descend method appearing at $q\to 1$ are precisely the Bethe equations, e.g., see Proposition 4.2 in  \cite{Pushkar:2016qvw}. In this Section we compute the corresponding saddle point equations for $q\to \zeta$ and show that they are given by the same Bethe equations with all variables raised to the power $p$, where $p$ is the order of the root of unity $\zeta$, see 
Corollary \ref{maincorrsee}.

\subsection{Integral representations of the vertex functions} Let $X$ be a Nakajima quiver variety associated to an oriented quiver $Q$.
For a vertex $i\in Q$ let $\mathcal{V}_{i}$ be the $i$-th tautological bundle and $\mathcal{W}_{i}$ be the framing tautological bundle  over $X$ (i.e., $\mathcal{W}_{i}$ is a trivial bundle).  We recall that the K-theory class
\begin{align} \label{polardef} 
P = \sum_{i\to j}\,\mathcal{V}^{*}_{i}\otimes\mathcal{V}_{j} + \bigoplus_{i\in Q} \mathcal{W}^{*}_{i}\otimes\mathcal{V}_{i} - \sum\limits_{i\in Q} \,  \mathcal{V}_{i}^{*} \otimes \mathcal{V}_{i}
\end{align}
where the first sum is over the edges of the quiver $Q$
connecting vertices $i$ and $j$,
is called the canonical polarization of $X$. This K-theory class represents a ``half'' of the tangent bundle, in the sense that:
\begin{align} \label{tanbun}
TX = P+\hbar^{-1}\, P^{*}. 
\end{align}
Let $L_i=\det \cV_{i}$, $i\in Q$ denote the set of tautological line bundles over $X$. We have
\begin{align} \label{polardet}
\det P = \bigotimes_{i\in Q}\, L^{n_i}_{i}, \ \ \ n_i \in \Z.
\end{align}
Following the notation of \cite{Aganagic:2017be} we denote by
$$
\zz_{\#} = \zz (-\hbar^{1/2})^{- \det P}
$$
the set of shifted K\"ahler variables.
More precisely, in components these shifts are equal:
$$
z_{\#,i} = z_i (-\hbar^{1/2})^{-n_i}, \ \ \ i \in Q.
$$

Let $\varphi(x,q)$ denote the $q$-analog of the reciprocal Gamma function:
\begin{align} \label{qgam}
\varphi(x,q) = \prod\limits_{i=0}^{\infty} (1-x q^i)
\end{align}
We extend this function to Laurent polynomials $\sum_{i} n_i a_i$, $n_i\in \mathbb{Z}$  by the rule (omitting the second argument of $\varphi$ for brevity)
\begin{align} \label{varphirule}
\Phi( \sum_{i} n_i x_i ) = \prod_{i} \varphi(x_i)^{n_i}.
\end{align}
Using (\ref{polardef}) and (\ref{tautinroots}) we can represent the polarization $P$ by a Laurent polynomial
in the Grothendieck roots of the tautological bundles $\xx$ and the equivariant parameters $\aa$, so that
$$
 P=\sum_{w \in N(P)}\, w, 
$$
where $N(P)$ denotes the Newton polygon of $P$ and $w$  are monomials in the variables $\xx$ and~$\aa$.

%We have
%\begin{align} \label{integra}
%\Phi((q-\hbar) P) = \prod\limits_{w \in N(P)} \, \Big(\dfrac{\varphi(q w)}{ \varphi(\hbar w) }\Big)^{n_w} \,.
%\end{align}

Let $\tau \in K_{T}(X)$ be a K-theory class represented by a Laurent polynomial in the Grothendieck roots~$\tau(\xx)$. We recall that the components of the vertex function of a Nakajima variety with descendant $\tau$ have the following integral representation:
\begin{align} \label{intrep}
V^{(\tau)}_i(\zz) = \int_{\gamma_i}  \Phi((q-\hbar) P) e(\zz,\xx)  \,\tau(\xx) \, \prod\limits_{a,b} \dfrac{d x_{a,b}}{ x_{a,b}}\,,
\end{align}
where $V^{(\tau)}_i(\zz)$ denotes the $i$-th component of the vertex function in the basis of the torus fixed points,
$$
e(\zz,\xx)  = \prod\limits_{i\in Q}\, \exp\left( \dfrac{\log(z_{\#,i})\log(L_i)}{\log(q)} \right),
$$
the function $\Phi((q-\hbar) P)$ is constructed from the polynomial $(q-\hbar) P$ by rule (\ref{varphirule}), and $\gamma_i$ is the contour defined by (A.12) in \cite{Aganagic:2017smx}.

\subsection{Integrand of (\ref{intrep}) near roots of unity.}
The infinite product (\ref{qgam}) converges for $|q|<1$. Thus, the integrand of (\ref{intrep}) is well defined only for $|q|<1$. However, when $q$ approaches $1 \in \mathbb{C}$, the divergent part can be separated as follows. 

\begin{Prop} \label{propasymp} The integrand of (\ref{intrep}) has the following form
\begin{align} \label{qintone}
\Phi((q-\hbar) P) e(\zz,\xx) =  \exp\left( -\dfrac{Y(\zz,\aa,\xx)}{1-q} \right) \cdot \star
\end{align}
where $\star$ stands for a function of all parameters regular at $q=1$ and 
\begin{align} \label{YangYang}
Y(\zz,\aa,\xx)= \sum\limits_{w\in N(P)}\, ({\rm Li}_2(w) -{\rm Li}_2(\hbar w)) + \sum\limits_{i \in Q}\, \log(z_{\#,i}) \log(L_i)\,,
\end{align}
where ${\rm Li}_2(w) = \sum_{m=1}^{\infty}\, \frac{x^m}{m^2}$ denotes the dilogarithm function. 
\end{Prop}

The function $Y(\zz,\aa,\xx)$ is known as the \textit{Yang-Yang} function in the literature on integrable systems \cites{Nekrasov:2009ui, Nekrasov:2009uh}.

\begin{proof}
We have
\begin{align}  \label{phifactor}
\dfrac{\varphi(q w)}{ \varphi(\hbar w) } =\prod\limits_{i=0}^{\infty} \dfrac{1- q w q^i}{1-\hbar w q^i}  = \exp\left( - \sum\limits_{m=1}^{\infty}\, \dfrac{(q^m-\hbar^m) w^m}{m(1-q^m)} \right)
\end{align}  
Note that 
\begin{align}  \label{asnstar}
\exp\left( - \sum\limits_{m=1}^{\infty}\, \dfrac{(q^m-\hbar^m) w^m}{m(1-q^m)} \right) = \exp\left( - \dfrac{1}{1-q}\sum\limits_{m=1}^{\infty}\, \dfrac{(1-\hbar^m) w^m}{m^2} \right)\cdot \star \end{align}
where 
$$
\star=\exp\left( - \sum\limits_{m=1}^{\infty}\, \dfrac{ w^m}{m} \left(\dfrac{q^m-\hbar^m}{1-q^m} - \dfrac{1-\hbar^m}{m (1-q)}\right) \right).
$$
Also note that $\star$ is regular at $q=1$:
$$
\left.\star\right|_{q=1} = \exp\left(  \sum\limits_{m=1}^{\infty}\, \dfrac{m+1+m \hbar^m-\hbar^m}{2 m^2}w^m \right)
$$
Thus, we may write:
$$
\exp\left( - \sum\limits_{m=1}^{\infty}\, \dfrac{(q^m-\hbar^m) w^m}{m(1-q^m)} \right) \sim \exp\left( - \dfrac{1}{1-q}\sum\limits_{m=1}^{\infty}\, \dfrac{(1-\hbar^m) w^m}{m^2} \right)  = \exp\left( - \dfrac{ \mathrm{Li}_{2}(w) - \mathrm{Li}_{2}(\hbar w)}{1-q}\right)
$$
where $\sim$ denotes equality modulo terms regular at $q=1$.  Similarly,
$$
\dfrac{1}{\log(q)} \sim -\dfrac{1}{1-q} 
$$
and hence
$$
e(z)  \sim \exp\left(-\dfrac{1}{1-q} \sum\limits_{i\in Q} \log(z_{\#,i}) \log(L_i)  \right)
$$
Combining these results for all factors gives the Yang-Yang function (\ref{YangYang}). 
\end{proof}

More generally, let $\zeta_p$ be a $p$-th primitive complex root of unity of order $p$. The divergent part of (\ref{intrep}) as $q$ approaches $\zeta_p$ can be separated as follows.
\begin{Prop}\label{secondassprop} The integrand of (\ref{intrep}) has the following form
\begin{align}  \label{rootinteg}
\Phi((q-\hbar) P) e(\zz,\xx)  = \exp\left( -\dfrac{Y(\zz^p,\aa^p,\xx^p)}{(1-q^p) p} \right) \cdot \star
\end{align}
where $Y(\zz^p,\aa^p,\xx^p)$ denotes the Yang-Yang function 
(\ref{YangYang}) in which all the variables are raised to the power $p$ and $\star$ denotes a function regular at $q=\zeta_p$. 
\end{Prop}

\begin{proof}
We note that as $q\to \zeta_p$ the divergent terms in the sum of (\ref{phifactor}) correspond to the summands with $m$ divisible by $p$. Separating these terms we obtain
$$
\dfrac{\varphi(q w)}{ \varphi(\hbar w) } =  \exp\left( - \sum\limits_{m=1}^{\infty}\, \dfrac{(q^{p m}-\hbar^{p m}) w^{p m}}{pm(1-q^{pm})} \right) \cdot \star
$$
where $\star$ stands for a factor regular at $q=\zeta_p$. 
Further, as in (\ref{asnstar}) we obtain
$$
 \exp\left( - \sum\limits_{m=1}^{\infty}\, \dfrac{(q^{p m}-\hbar^{p m}) w^{p m}}{pm(1-q^{pm})} \right) \sim  \exp\left( - \dfrac{1}{(1-q^p) p}\sum\limits_{m=1}^{\infty}\, \dfrac{(1-\hbar^{p m}) w^{p m}}{m^2} \right)
$$
where $\sim$ now denotes equality up to multiples regular at $q=\zeta_p$. Next,
$$
e(z)  = \prod\limits_{i\in Q}\, \exp\left( \dfrac{\log(z_{\#,i}^p)\log(L_i^p)}{p \log(q^p)} \right)
$$
From which we see that 
$$
e(z) \sim \exp\left( -\dfrac{1}{(1-q^p)p}\, \sum_{i\in Q}\, \log(z_{\#,i}^p)\log(L_i^p)  \right)
$$
Combining these results for all the factors proves the Proposition. 
\end{proof}

\subsection{Asymptotics of vertex functions via steepest descend method}

By Proposition \ref{propasymp} as $q\to 1$ the descendent vertex functions are given by the integrals 
$$
V^{(\tau)}_i(\zz,\aa,q) = \int_{\gamma}  \exp\left(\frac{Y(\zz,\aa,\xx)}{(1-q)}\right)\, \cdot\star\cdot \dfrac{d \xx}{\xx}
$$
where $\star$ denotes terms regular at $q=1$. Using the steepest descent method with small parameter $\epsilon= 1-q$, we find that 
\begin{equation} \label{firstasym}
V^{(\tau)}_i(\zz,\aa,q) =  \exp\Big(  -\dfrac{\lambda_i(\zz,\aa)}{1-q}\Big) \cdot \star
\end{equation}
where 
\begin{equation} 
\label{lamfun}
\lambda_i(\zz,\aa) = Y(\zz,\aa,\xx_i(\zz,\aa))
\end{equation}
is the value of the Yang-Yang function at its critical point $\xx_i(\zz,\aa)$ on the contour $\gamma$ and $\star$ denotes terms regular at $q=1$.

The critical points of the Yang-Yang function are the solutions to equations
\begin{equation} \label{critval}
\frak{B}_{i,k}(\zz,\aa,\xx):=x_{i,k} \dfrac{\partial}{\partial x_{i,k}} \, Y(\zz,\aa,\xx) =0, \ \ i \in Q, \ \ k =1,\dots, \textrm{rk}(\cV_i). 
\end{equation}
In Section \ref{betheeqseq} we recall an explicit form of these equations, in particular, we recall that (\ref{critval}) is equivalent to a system of algebraic equations also known as ``Bethe equations'' in the theory of integrable systems.

Next, assume that $q$ approaches a complex root of unity $\zeta_p$
of order $p$. By Proposition \ref{secondassprop}, the descendant vertex functions are given by the integrals 
$$
V^{(\tau)}_i(\zz,\aa,q) = \int_{\gamma}  \exp\left(\frac{Y(\zz^p,\aa^p,\xx^p)}{(1-q^p) p}\right)\, \cdot\star\cdot \dfrac{d \xx}{\xx}
$$
where $\star$ is a function regular at $q=\zeta_p$. 
In this case, the critical points for the steepest descend method are given by solutions to the following equations
$$
x_{i,k} \dfrac{\partial}{\partial x_{i,k}} Y(\zz^p,\aa^p,\xx^p) = p (x_{i,k})^p \dfrac{\partial}{\partial (x_{i,k})^p} Y(\zz^p,\aa^p,\xx^p) = 0
$$
i.e., the critical point equations are
\begin{equation}
\label{critvalp}
\frak{B}_{i,k}(\zz^p,\aa^p,\xx^p)=0, \ \ i \in Q, \ \ k =1,\dots, \textrm{rk}(\cV_i).
\end{equation}
with the same $\frak{B}_{i,k}$ as in (\ref{critval}).

Note that if $\xx_i(\zz,\aa)$ is a solution of critical point equations (\ref{critval}) then $\xx_i(\zz^p,\aa^p)^{1/p}$ solves the critical point equations (\ref{critvalp}). Combining this together, we find that near $q=\zeta_p$ the vertex functions have the following behaviour: 
$$
V^{(\tau)}_i(\zz,\aa,q)  = \exp\left( -\dfrac{\lambda_i(\zz^p,\aa^p)}{(1-q^p)p} \right) \cdot \star\,,
$$
where $\star$ denotes a power series in $\zz$ with coefficients which do not have poles at $q=\zeta_p$.  Note that 
$
\lambda(\zz^p,\aa^p)= Y(\zz^p,\aa^p,\xx_i(\zz^p,\aa^p)) 
$
is the critical value function (\ref{lamfun}) in which all parameters are raised to $p$-th power $\zz\to \zz^p$, $\aa\to \aa^p$.

\begin{Cor} \label{corrqsq}
Let $V^{(\tau)}_{i}\left(\zz^p,\aa^p,q^{p^2}\right)$ be the vertex functions in which all K\"ahler parameters $\zz$ and all equivariant parameters $\aa$ are raised to power $p$, and $q$ is raised to power $p^2$. This function has the following form:
$$
V_i^{(\tau)}(\zz^p,\aa^p,q^{p^2}) = \exp\left( -\dfrac{\lambda_i(\zz^p,\aa^p)}{(1-q^p)p} \right)\cdot \star
$$
where
$\star$ is a power series in $\zz$ with coefficients which do not have poles at $q=\zeta_p$.  
\end{Cor}
\begin{proof}
Note that when $q\to \zeta_p$ we have $q^{p^2}\to 1$. Thus, by (\ref{firstasym}) we have :
$$
V_i^{(\tau)}(\zz^p,\aa^p,q^{p^2}) \to \exp\left( -\dfrac{\lambda_i(\zz^p,\aa^p)}{(1-q^{p^2})} \right)\cdot \star
$$
The Corollary holds since 
$$
\dfrac{1}{1-q^{p^2}} = \dfrac{1}{(1-q^{p})(1+q^p+\dots+q^{p(p-1)})} \stackrel{q\to \zeta_p}{\to}\dfrac{1}{(1-q^{p})p}
$$
\end{proof}
We recall that the coefficients of the power series $V^{(\tau)}_i(\zz,\aa,q)$ have poles in $q$ located at the roots of unity.  The last two corollaries imply that the coefficients of the power series 
$$
\dfrac{V_i^{(\tau)}(\zz^p,\aa^p,q^{p^2})}{V^{(\tau')}_i(\zz,\aa,q)}
$$
do not have poles at $q$ given by the primitive roots of unity of order $p$ for any choices of descendants insertion $\tau$ and $\tau'$, i.e., these poles are canceled in the ratio. 

\section{Bethe Equations \label{betheeqseq}}

\subsection{Bethe Equations for Yang-Yang Functions} 
The critical points of the Yang-Yang functions $Y(\zz,\aa,\xx)$ are determined by the equations
\begin{align} \label{betheeqn1}
x_{i,k} \dfrac{\partial Y(\zz,\aa,\xx)}{\partial x_{i,k} } =0, \ \ i \in Q, \ \ k =1,\dots, \textrm{rk}(\cV_i). 
\end{align}
In the theory of integrable spin chains these equations appear as {\it Bethe Ansatz equations}.
These equations can be written in the following convenient form. Let us define the following function
$$
\hat{a}(x) = x^{1/2}-x^{-1/2}
$$
and extend it by linearity to Laurent polynomials with integral coefficients by the rule
$$
\hat{a}\left(\sum_{i} \, m_i x_{i}\right) =  \prod_{i}\, \hat{a}(x_i)^{m_i},
$$
where $x_i$ denote monomials in the equivariant parameters and $x_{i,k}$. 

Let $TX$ be the K-theory class of the tangent space (\ref{tanbun}), written as a Laurent polynomial in Grothendieck roots $\xx= \{x_{i,j}\}$ and the equivariant parameters  for the canonical choice of polarization (\ref{polardef}). The following description of the Bethe equations was obtained in  \cite{Aganagic:2017be}:
\begin{Prop}[Proposition 9,  \cite{Aganagic:2017be}] \label{betheqprop}
Bethe Ansatz equations (\ref{betheeqn1}) have the following form
\begin{align} \label{betheeq}
\hat{a}\left(x_{i,k} \dfrac{\partial}{\partial x_{i,k}} TX \right) =z_i, \ \ i \in Q, \ \ k=1,\dots, rk(\cV_i).
\end{align}
\end{Prop}
\begin{proof}
For the Yang-Yang function (\ref{YangYang}) we write 
$$
Y(\zz,\aa,\xx)= Y_1(\zz,\aa,\xx) + Y_{2}(\zz,\aa,\xx)
$$
with 
$$
Y_1(\zz,\aa,\xx)=\sum\limits_{w\in N(P)}\, ({\rm Li}_2(w) -{\rm Li}_2(\hbar w)), \ \ \ Y_{2}(\xx,\aa,\zz)= \sum\limits_{i \in Q}\, \log(z_{\#,i}) \log(L_i)
$$
where $L_i = \det  \cV_{i}=  \prod_{j}\,{x_{i,j}}$.
Let 
$$
W(w) = {\rm Li}_2(w) -{\rm Li}_2(\hbar w)
$$
be a summand of $Y_1(\zz,\xx)$. Since $w$ is a monomial in the Grothendieck roots $x_{i,k}$ we compute
\begin{align} \label{ap1}
w \dfrac{\partial W(w)}{\partial x_{i,k}}  = \dfrac{\partial w}{\partial x_{i,k}} \sum\limits_{m=1}^{\infty}\, \dfrac{(1-\hbar^m) w^m}{m} =  -\dfrac{\partial w}{\partial x_{i,k}}  \log\left(\dfrac{1-w}{1-\hbar w}\right).
\end{align}
A straightforward calculation gives
\begin{align} \label{ap2}
\log\left(\hat{a}\left(x_{i,k} \dfrac{\partial}{\partial x_{i,k}} (w+\frac{1}{\hbar w}) \right)\right) =\dfrac{x_{i,k}}{w} \dfrac{\partial w}{\partial x_{i,k}}   \log(-\hbar^{1/2}) +\dfrac{x_{i,k}}{w} \dfrac{\partial w}{\partial x_{i,k}} \log\left(\frac{1-w}{1-\hbar w}\right).
\end{align}
Combining (\ref{ap1}) and (\ref{ap2}) we get
\begin{align} \label{termelem}
x_{i,k} \dfrac{\partial W(w)}{\partial x_{i,k}} = \dfrac{x_{i,k}}{w} \dfrac{\partial w}{\partial x_{i,k}} \log(-\hbar^{-1/2})  -  \log\left(\hat{a}\left(x_{i,k} \dfrac{\partial}{\partial x_{i,k}} (w+\frac{1}{\hbar w}) \right)\right)\,.
\end{align}
Wring (\ref{polardet}) as
$$
\det P  = \prod\limits_{i\in Q} \left(\prod_{k} x_{i,k}\right)^{n_i},
$$
we see that
$$
x_{i,k} \dfrac{\partial \det P }{\partial x_{i,k}} = n_i \det P, 
$$
but from $\det P  = \prod_{w\in N(P)} w$ we obtain 
$$
x_{i,k} \dfrac{\partial \det P }{\partial x_{i,k}}  =   \det P \sum_{w\in N(P)} \dfrac{x_{i,k}}{w} \dfrac{\partial w }{\partial x_{i,k}}. 
$$
Comparing the last two expressions we obtain 
\begin{align} \label{nnums}
n_i = \sum_{w\in N(P)} \dfrac{x_{i,k}}{w} \dfrac{\partial w }{\partial x_{i,k}} \,.
\end{align}
From (\ref{tanbun}) for the virtual tangent space we have
$$
TX = \sum\limits_{w\in N(P)} \, \left(w +\dfrac{1}{\hbar w}\right). 
$$
Since
$$
Y_{1}(\zz,\aa,\xx)  =  \sum\limits_{w\in N(P)}\, W(w)
$$
 summing (\ref{termelem}) over $w$ gives
$$
x_{i,k} \dfrac{\partial Y_{1}(\zz,\aa,\xx)}{\partial x_{i,k}} = \log(-\hbar^{-1/2})   \left( \sum\limits_{w}  \dfrac{x_{i,k}}{w} \dfrac{\partial w}{\partial x_{i,k}}\right) -  \log\left(\hat{a}\left(x_{i,k} \dfrac{\partial}{\partial x_{i,k}} TX \right)\right)
$$
or, by (\ref{nnums}) we have
$$
x_{i,k} \dfrac{\partial Y_{1}(\zz,\aa,\xx)}{\partial x_{i,k}} = \log(-\hbar^{-1/2})   n_i -  \log\left(\hat{a}\left(x_{i,k} \dfrac{\partial}{\partial x_{i,k}} TX \right)\right).
$$
Note also that
$$
x_{i,k} \dfrac{\partial Y_{2}(\zz,\aa,\xx)}{\partial x_{i,k}}=\log(z_{i,\#}) = \log(z_{i}) - \log(-\hbar^{1/2}) n_i.
$$
Overall we obtain 
$$
x_{i,k} \dfrac{\partial Y(\zz,\aa,\xx)}{\partial x_{i,k}} =x_{i,k} \dfrac{\partial Y_{1}(\zz,\aa,\xx)}{\partial x_{i,k}}+ x_{i,k} \dfrac{\partial Y_{2}(\zz,\aa,\xx)}{\partial x_{i,k}} =  -  \log\left(\hat{a}\left(x_{i,k} \dfrac{\partial}{\partial x_{i,k}} TX \right)\right) + \log(z_i). 
$$
Therefore the Bethe equations (\ref{betheeq}) are
$$
 -  \log\left(\hat{a}\left(x_{i,k} \dfrac{\partial}{\partial x_{i,k}} TX \right)\right) + \log(z_i) =0,
$$
which, after exponentiation finishes the proof of the Proposition. 
\end{proof}

\begin{Cor}  \label{maincorrsee}
The non-zero\footnote{By the chain rule, $\xx=0$ is always a critical point for $p>1$.} critical points $\hat{\xx}$ of $Y(\zz^p,\aa^p,\xx^p)$ from Corollary \ref{corrqsq} solve Bethe equations (\ref{betheeq}) with all variables raised to the power $p$
\begin{align}\label{frobbeth}
\left.\hat{a}\left(x_{i,k} \dfrac{\partial}{\partial x_{i,k}} TX \right)\right|_{x_{l,m}\to x_{l,m}^p,a_{j}\to a_j^{p},\hbar\to \hbar^{p}} =z^p_i
\end{align}
\end{Cor}

\subsection{Example} 
As an example, let us consider $X=T^{*} \mathrm{Gr(k,n)}$ be the cotangent bundle to the Grassmannian of $k$ hyperplanes in $\mathbb{C}^n$. The corresponding tautological bundles have the form
$$
\cV = x_1+\dots +x_k, \ \ \ \cW = a_1+\dots + a_n
$$
The polarization (\ref{polardef}) equals
$$
P= \cW^{*} \otimes \cV  - \cV^{*}\otimes \cV
$$
or 
$$
P=  \sum_{j=1}^{n} \sum_{i=1}^{k} \dfrac{x_i}{a_j} - \sum_{i,j=1}^{k}\, \dfrac{x_{i}}{x_{j}}\,.
$$
Thus for the virtual tangent space (\ref{tanbun}) we obtain
$$
TX = \sum_{j=1}^{n} \sum_{i=1}^{k} \left(\dfrac{x_i}{a_j} + \dfrac{a_j}{\hbar x_i}\right) - \sum_{i,j=1}^{k}\, \left( \dfrac{x_{i}}{x_{j}} + \dfrac{x_{j}}{\hbar x_{i}} \right) \,.
$$
Then
$$
x_{m} \dfrac{\partial}{\partial x_{m}} TX =  \sum\limits_{j=1}^{n} \left(\frac{x_m}{a_j}-\frac{a_j}{\hbar x_m}\right) - (1+\hbar^{-1}) \sum\limits_{i=1}^{k}\left(\dfrac{x_m}{x_i}-\dfrac{x_i}{x_{m}}\right)
$$
and after simplification, the equation
$$
\hat{a}\left(x_{m} \dfrac{\partial}{\partial x_{m}} TX \right) = z 
$$
takes the form
\begin{align} \label{sl2spin}
\prod\limits_{j=1}^{n} \dfrac{x_m-a_j}{a_j-\hbar x_m}\prod\limits_{i=1}^{k} \,\dfrac{x_j-x_m \hbar}{\hbar x_j -  x_m} = z \hbar^{-n/2}, \qquad  m=1,\dots,k. 
\end{align}
Equations (\ref{sl2spin}) are the well known Bethe Ansatz equations for the $U_{\hbar}(\widehat{\mathfrak{sl}}_2)$ XXZ spin chain on $n$ sites in the sector with $k$ excitations \cite{Reshetikhin:2010si}. The critical point equations (\ref{frobbeth}) describing the asymptotic $q\to \zeta_p$
then take the form:
\begin{align}  \label{pbethe}
\prod\limits_{j=1}^{n} \dfrac{x^{p}_m-a^{p}_j}{a^{p}_j-\hbar^{p} x^{p}_m}\prod\limits_{i,m=1}^{k} \,\dfrac{x^{p}_j-x^{p}_m \hbar^{p}}{\hbar^{p} x^{p}_j -  x^{p}_m} = z^{p} \hbar^{-n p /2}, \qquad m=1,\dots,k. 
\end{align}
For instance, when $k=1,  n=2$, corresponding to $X=T^{*} \mathbb{P}^1$ the sole Bethe equation (\ref{sl2spin}) reads
$$
(x-a_1) (x-a_2) - z \hbar^{-1} (x-a_1 \hbar) (x-a_2 \hbar) =0\,,
$$
and (\ref{pbethe}) has the form
$$
(x^{p}-a^{p}_1)(x^{p}-a^{p}_2)- z^p \hbar^{-p } (x^{p}-a^{p}_1 \hbar^{p})(x^{p}-a^{p}_2 \hbar^{p}) =0\,.
$$
The last two equations are precisely the characteristic equations for the operators $\mathcal{M}_{\cL}(z)$ and $\mathcal{M}_{\cL,\zeta_p}(z)$ for $X=T^{*} \mathbb{P}^1$ discussed in Section \ref{examplesec} (up to an obvious difference in notations $\hbar \to \hbar^{-2}$).

\section{Asymptotics of Fundamental Solutions}
\subsection{Integral representation for the fundamental solution matrix} 
Let us recall that the fundamental solution matrix
for the QDE of a Nakajima variety (\ref{QDEintro}) has the following integral representation \cite{Aganagic:2017be,Aganagic:2017smx,2022arXiv220501596D}. In the K-theoretic stable basis of $K_{T}(X)$ the components of the fundamental solution matrix have the form:
\begin{equation} 
\label{intsolmat}
\Psi_{i,j}(\zz,\aa,q) = \int\limits_{\gamma_{j}} \,\, \Phi((q-\hbar) P) e(\zz,\xx)\, S_{i}(\xx,\aa) \, \prod\limits_{a,b} \dfrac{d x_{a,b}}{ x_{a,b}}\,,
\end{equation}
where $\gamma_{j}$ is the contour in the space of variables $\xx$ defined by (A.12) in \cite{Aganagic:2017smx}. 
The function $S_{i}(\xx,\aa)$ represents the class of the K-theoretic stable envelopes of the torus fixed point $i$\footnote{$S_{i}(\xx,\aa)$ also appears in literature under the name  ``abelianized stable envelope'' or ``weight function''.}. 

In the terminology of Section 
\ref{section2}, the components of the fundamental solution matrix are the vertex functions (\ref{intrep}) with the descendants given by $S_i(\xx,\aa)$:
\begin{align} \label{psiasvert}
\Psi_{i,j}(\zz,\aa,q) = V^{(S_{i}(\xx,\aa))}_j(\zz,\aa,q).
\end{align}
From the Corollary \ref{corrqsq} we obtain that the fundamental solution matrix has the following form
\begin{align} \label{firstasint}
\Psi_{i,j}(\zz,\aa,q)  
=   \exp\left(- \dfrac{\lambda_j(\zz^p,\aa^p)}{(1-q^p) p}  \right) \psi_{i,j}(\zz,\aa,q)  
\end{align}
where $\psi_{i,j}(\zz,\aa,q)$ are certain power series in $\zz$ with coefficients which do not have poles in $q$ at roots of unity of order $p$.

\subsection{Inverse of the fundamental solution matrix}
The inverse of the fundamental solution (\ref{intsolmat}) matrix has the following form 
\begin{equation} \label{secondasimpt}
(\Psi(\zz,\aa,q)^{-1})_{i,j}  =  \exp\left(\dfrac{\lambda_i(\zz^p,\aa^p)}{(1-q^p) p}  \right) \phi_{i,j}(\zz,\aa,q) 
\end{equation}
for some power series $\phi_{i,j}(\zz,\aa,q)$  in $\zz$ with coefficients which are regular at $q$ given by unit roots of order $p$. 
Note that this formula does not follow immediately from (\ref{firstasint}) since  the inverse of the matrix $\psi_{i,j}(\zz,\aa,q)$ may have poles in $q$ at these points.  To show that this is not the case we use a geometric argument, which is outlined below.

Recall that the matrix $\Psi_{i,j}(\zz,\aa,q)$, also known as the {\it capping matrix} \cite{Okounkov:2015aa},  is a partition function counting equivariant quasimaps from $\mathbb{P}^1$  to a Nakajima variety $X$ with relative boundary conditions at $p_1=0 \in \mathbb{P}^1$ and non-singular conditions at $p_2=\infty \in \mathbb{P}^1$, see Fig.~\ref{figcap} for the diagrammatic representation of these boundary conditions adopted in \cite{Okounkov:2015aa}. 

\begin{figure}[h!]
    \centering
    \includegraphics[width=0.3\textwidth]{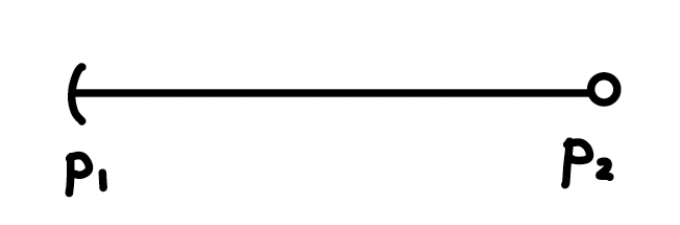}
    \caption{Diagram representing the capping operator}
    \label{figcap}
\end{figure}
\noindent
The parameter $q$ features in this geometric setting as the equivariant parameter given by the character of the tangent space $T_{p_1} \mathbb{P}^1$, for torus $\mathbb{C}^{\times}_q$, which acts on $\mathbb{P}^1$ via rotation so that $(\mathbb{P}^1)^{\mathbb{C}^{\times}_q}=\{p_1=0,p_2=\infty\}$.
Similarly, (see Section 7.1 in \cite{Okounkov:2015aa})  the {\it glue matrix} 
$$
{\bf G}(z) \in  K_{T}(X)^{\otimes 2}(\zz)
$$
is defined as a partition function counting equivariant quasimaps from $\mathbb{P}^1$ with relative boundary conditions at $p_1=0 \in \mathbb{P}^1$ and relative conditions at $p_2=\infty \in \mathbb{P}^1$, see Fig.~\ref{figglue}. 
\begin{figure}[h!]
    \centering
    \includegraphics[width=0.3\textwidth]{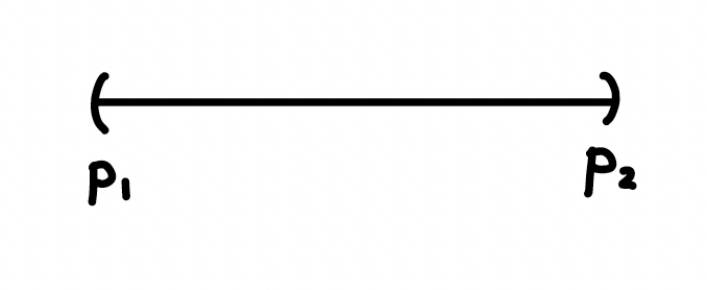}
    \caption{Diagram representing the glue matrix}
    \label{figglue}
\end{figure}
The glue matrix ${\bf G}(z)$ is defined for {\it proper} evaluation maps, in particular, its matrix elements live in the integral $K$-theory and, moreover, 
are independent on the equivariant parameter~$q$.  

Using the $\mathbb{C}^{\times}_{q}$ localization exactly as in Sections 7.1.6 - 7.1.12 of \cite{Okounkov:2015aa}, one can separate the equivariant count for ${\bf G}(z)$  into contributions from $p_1$ and $p_2$ as in Fig. \ref{locpic} below\footnote{This should not be confused with the degeneration/gluing formulas, which are of different nature.}.  Let us briefly recall the idea of the $\mathbb{C}^{\times}_{q}$-localization: if a quasimap from $\mathbb{P}^1$ is $\mathbb{C}^{\times}_{q}$-equivariant, then it can not have singularities in $\mathbb{P}^1\setminus \{0,\infty\}$. Since for   ${\bf G}(z)$ we also impose relative conditions at $0$ and $\infty$ such quasimap has no singularities on the entire  $\mathbb{P}^1$ and therefore is constant. We conclude that the partition function ${\bf G}(z)$ factors to a product of two partition functions -- one coming from the bubbles at $0$ and another coming from the bubbles at $\infty$. The former factor is the usual capping matrix. The latter, from the symmetry of the picture, is given by the transposition of the capping matrix together with the substitution $q\to q^{-1}$.  If $q$ is the character of the tangent space at $T_{p_1}\mathbb{P}^1$ then the character of $T_{p_2}\mathbb{P}^1$ is $q^{-1}$,  which explains the substitution $q\mapsto q^{-1}$. 

\begin{figure}[h!]
    \centering
    \includegraphics[width=0.5\textwidth]{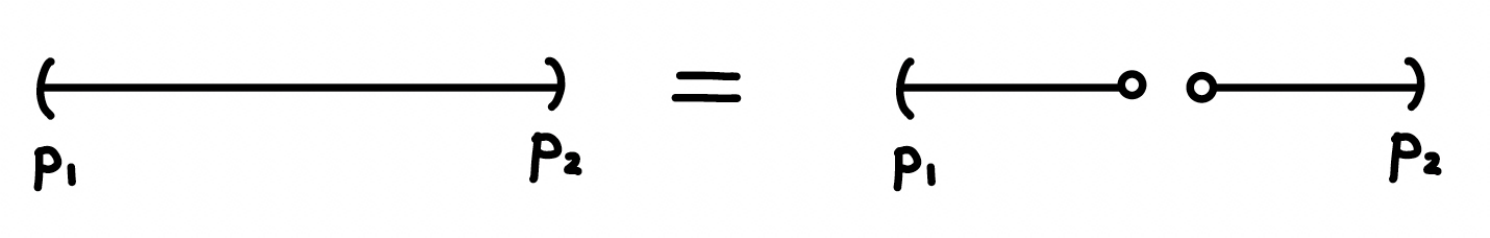}
    \caption{Computing glue matrix by $\mathbb{C}^{\times}_q$ - localization}
    \label{locpic}
\end{figure}
\noindent
This gives us the following formula for the gluing matrix in terms of the capping matrices:
$$
{\bf G}(z) = \Psi(\zz,\aa,q) \, \Psi(\zz,\aa,q^{-1})^{t}
$$
From the last formula we obtain that 
$$
\Psi(\zz,\aa,q)^{-1} = \Psi(\zz,\aa,q^{-1})^{t} {\bf G}(z)^{-1} 
$$
and now (\ref{secondasimpt}) follows from (\ref{firstasint}).

\subsection{Frobenius matrix}
Using the above asymptotics we find:
\begin{Thm} \label{ratthm}
The operator 
\begin{align} \label{ratoffunds}
\Psi(z,a,q)\, \Psi\left(z^p,a^p,q^{p^2}\right)^{-1}\,
\end{align}
has no poles in $q$ at the roots of unity of order $p$.  
\end{Thm}
\begin{proof}
By (\ref{firstasint}) we have
\begin{align} 
\Psi_{i,j}(\zz,\aa,q)  
=   \exp\left(- \dfrac{\lambda_j(\zz^p,\aa^p)}{(1-q^p) p}  \right) \psi_{i,j}(\zz,\aa,q) 
\end{align}
Similarly, from (\ref{secondasimpt}) for $p=1$ we obtain
$$
\Psi_{i,j}(\zz,\aa,q)^{-1}  =  \exp\left(\dfrac{\lambda_i(\zz,\aa)}{(1-q)}  \right) \phi_{i,j}(\zz,\aa,q) 
$$
substituting $\aa\mapsto \aa^p$, $\zz\mapsto \zz^p$ and $q\mapsto q^{p^2}$ to this equality we  also obtain
\begin{align} 
\Psi_{i,j}(\zz^p,\aa^p,q^{p^2})^{-1}  =  \exp\left(\dfrac{\lambda_i(\zz^p,\aa^p)}{(1-q^p) p}  \right) \phi'_{i,j}(\zz,\aa,q)
\end{align}
with some $\phi'_{i,j}(\zz,\aa,q)$ regular at $p$-th unit roots. Thus we obtain:
$$
\Psi(\zz,\aa,q)\, \Psi\left(\zz^p,\aa^p,q^{p^2}\right)^{-1}= \psi(\zz,\aa,q)\, \phi'(\zz,\aa,q)
$$
which does not have poles as desired. 
\end{proof} 

\vspace{3mm}

Let $\mathcal{M}_{\cL}(\zz,\aa)$  and $\mathcal{M}_{\cL,\zeta_p}(\zz,\aa)$ be the operators defined by (\ref{qprod}) and (\ref{Qpoper}) respectively. Let $\mathcal{M}_{\cL}(\zz^p,\aa^p)$ denotes the operator $\mathcal{M}_{\cL}(\zz,\aa)$ with all parameters raised to the power $p$: $\zz^p=(z^p_1,\dots,z^p_n)$ and $\aa^p=(a^p_1,\dots, a^p_m)$.

\begin{Thm} \label{Th:Isopectral}
 The operators $\mathcal{M}_{\cL}(\zz^p,\aa^p)$ and  $\mathcal{M}_{\cL,\zeta_p}(\zz,\aa)$ are conjugate to each other.
\end{Thm}
\begin{proof}
The fundamental solution matrix $\Psi(\zz,\aa,q)$ satisfies the $q$-difference equation
\begin{align} \label{QDEref}
\Psi(\zz q^{\cL},\aa,q) \cL = \mathbf{M}_{\cL}(\zz,\aa,q) \Psi(\zz,\aa,q)\,.
\end{align}
Iterating this equation $p$ times we obtain
\begin{align} \label{QDEiter}
\Psi(\zz q^{p \cL},\aa,q) \cL^p = \mathbf{M}_{\cL^p}(\zz,\aa,q) \Psi(\zz,\aa,q),
\end{align}
where 
\begin{equation}\label{Mlpopp}
\mathbf{M}_{\cL^p}(\zz,\aa,q)  =\mathbf{M}_{\cL}(\zz  q^{(p-1)\cL},\aa,q) \cdots \mathbf{M}_{\cL}(\zz q^{\cL} ,\aa,q)\mathbf{M}_{\cL}(\zz,\aa,q).
\end{equation}
Replacing all K\"ahler $\zz$ and equivariant $\aa$ parameters with their $p$-th powers $\zz^p$ and $\aa^p$, and $q$ with $q^{p^2}$ in (\ref{QDEref}) we obtain
\begin{align} \label{QDEp}
\Psi((\zz q^{p\cL})^p,\aa^p,q^{p^2}) \cL^p = \mathbf{M}_{\cL}(\zz^p,\aa^p,q^{p^2}) \Psi(\zz^p,\aa^p,q^{p^2}).
\end{align}
Dividing (\ref{QDEiter}) by (\ref{QDEp}) we obtain
\begin{align}\label{eq:FrobIntertwiner}
    \Psi(\zz q^{p \cL},\aa,q)\, &\Psi\left((\zz q^{p\cL})^p,\aa^p,q^{p^2}\right)^{-1}\cr
&= 
 \mathbf{M}_{\cL^p}(\zz,\aa,q)\,\cdot \Psi(\zz,\aa,q) \Psi(\zz^p,\aa^p,q^{p^2})^{-1} \,\cdot \mathbf{M}_{\cL}(\zz^p,\aa^p,q^p)^{-1}. 
\end{align}
By Theorem \ref{ratthm} the operator
$$
\mathsf{F}(\zz,\aa,\zeta_p) = \Psi(\zz,\aa,q)\,\Psi\left(\zz^p,\aa^p,q^{p^2}\right)^{-1}\Big\vert_{q=\zeta_p}
$$
is well defined. Thus, specializing (\ref{eq:FrobIntertwiner}) at $q=\zeta_p$ we obtain:
$$
\mathsf{F}(\zz,\aa,\zeta_p) =   \mathcal{M}_{\cL,\zeta_p}(\zz,\aa)\mathsf{F}(\zz,\aa,\zeta_p) \mathcal{M}_{\cL}(\zz^p,\aa^p)^{-1}\,,
$$
where
$$
\mathcal{M}_{\cL,\zeta_p}(\zz,\aa) = \left.\mathbf{M}_{\cL^p}(\zz,\aa,q)\right|_{q=\zeta_p}, \qquad \mathcal{M}_{\cL}(\zz^p,\aa^p) = \left.\mathbf{M}_{\cL}(\zz^p,\aa^p,q^p)\right|_{q=1}.
$$
Rearranging the terms we arrive at
$$
\mathsf{F}(\zz,\aa,\zeta_p) \mathcal{M}_{\cL}(\zz^p,\aa^p) \mathsf{F}(\zz,\aa,\zeta_p)^{-1} =   \mathcal{M}_{\cL,\zeta_p}(\zz,\aa)\,,
$$
from where the Theorem follows.
\end{proof}

\begin{Cor}
Let $\{\lambda_1(\zz,\aa),\lambda_2(\zz,\aa), \dots\}$ be the set of eigenvalues of $\mathcal{M}_{\cL}(\zz,\aa)$ then the eigenvalues of $\mathcal{M}_{\cL \,\zeta_p}(\zz,\aa)$ are given by the set
$\{\lambda_1(\zz^p,\aa^p),\lambda_2(\zz^p,\aa^p),\dots\}$
where 
$\lambda_i(\zz^p,\aa^p)$
is the eigenvalue $\lambda_i(\zz,\aa)$ in which all K\"ahler variables $z_i$ and equivariant variables $a_i$ are substituted by $z^p_i$ and $a^p_i$ respectively. 

\end{Cor}

\subsection{Example}
Consider $X=T^*\mathbb{P}^0$ as the simplest example. 
The fundamental solution, which is a scalar function in this case, reads
$$
\Psi(z,\hbar,q)=\frac{\varphi(\hbar z)}{\varphi(z)} = \prod_{i=0}^{\infty}\, \dfrac{1-\hbar z q^i}{1-z q^i} = \exp\left(\sum\limits_{m=1}^{\infty}\, \dfrac{1-\hbar^m}{1-q^m}\frac{z^m}{m} \right)
$$
which satisfies the QDE:
$$
\Psi(z q ,\hbar,q) = \frac{1-z}{1-\hbar z}\Psi(z,\hbar,q)\,.
$$
Thus
$$
\Psi(z,\hbar,q) \Psi\left(z^p,\hbar^p,q^{p^2}\right)^{-1}=\exp\left(\sum_{m=1}^\infty \frac{(1- \hbar^{m} )z^{m}}{m(1- q^m)}-\frac{(1-\hbar^{pm} )z^{pm}}{m(1-q^{p^2 m})}\right).
$$
The poles in $q$ at roots of unity of order $p$ cancel out and taking the limit $q\to \zeta_p$ we obtain
$$
\mathsf{F}(z,\zeta_p) = \exp\left(\sum\limits_{m=1}^{\infty}\, \dfrac{1-\hbar^m}{m} z^m \delta_m \right),
$$
where
$$
\delta_m = \left\{\begin{array}{cc}
\frac{1}{1-\zeta_p^m}, & p\not{\mid} \, m\\
\frac{1-p}{2}, & p\mid m
\end{array}\right. 
$$

\begin{Rem}
The transformation $\aa\mapsto \aa^p$, $\zz\mapsto \zz^p$, $q\mapsto q^{p^2}$
may not look symmetric. This visible asymmetry may have the following simple interpretation in terms of $q$-quantized coordinate rings: let is consider an algebra of $q$-commuting coordinates $X$, $Y$ with relations 
$$
X Y =q Y X
$$
The Frobenius transformation acts on the coordinates by $X\to X^p$, $Y\to Y^p$, which are now subject to relations 
$$
X^p Y^p = q^{p^2} Y^p X^p
$$
which means that $q\to q^{p^2}$ under the Frobenius. In this context, $q$-deformations of Frobenius structures were also investigated by M.\,Kontsevich, as communicated to the second author during his visit to IHES in the Summer of 2025.
\end{Rem}

\begin{Rem}
Since the analysis of the preceding sections relies on the integral representations of vertex functions from \cite{Aganagic:2017smx} we would like to comment on these integral formulae.

The vertex functions of Nakajima quiver varieties, come with natural power series representations: 
\begin{align} \label{powserrep}
V(\zz,\aa) = \sum_{\degs \in H_2(X,\mathbb{Z})_{\textrm{eff}}} {\zz^{\degs}} I({q^{\degs}} , q, \aa) 
\end{align}
where the sum runs over the cone of effective curves in $X$ and  $I(\xx , \aa)$ is a certain product of $q$-gamma functions constructed from the underlying quiver in a combinatorial way. These power series representations come from computing the vertex function via equivariant localization, i.e., the coefficients represent the contributions of the torus fixed points on the quasimap moduli spaces. The purpose of the integral formula of \cite{Aganagic:2017smx} is to merely represent this power series as a summation over the residues of a global function. 
Thus, it may be more natural to derive the asymptotic behavior of the vertex function directly from the series representation (\ref{powserrep}) using the multidimensional Laplace method for sums. According to this method, the asymptotics are governed by the critical points of the phase function $\log I(\xx,q,\aa)$ as $q\to\zeta_p$. The resulting computations are equivalent to the critical-point analysis of the the preceding sections.

The analytic aspects of the integral representations of vertex functions have been studied extensively in the literature on integrable systems. Halan\"as and Ruijsenaars \cite{HR1} (see section 2) constructed an analytic formula for the eigenfunction of the Ruijsenaars-Macdonald operators which in our geometric language corresponds to the vertex functions for the cotangent bundle to the complete flag variety. The infinite products in the integrand can be analytically continued to the double sine functions.
More generally, the correspondence between integrability of the Rujisenaars system and geometry of quiver varieties was addressed in \cite{Koroteev:2018a} (see Theorem 4.8). 
 
\end{Rem}

\section{$p$-Curvature and Frobenius}\label{Sec:pCurvFrob}
In this final section, we discuss a reduction of the isospectrality Theorem \ref{Th:Isopectral} to a field of finite characteristic. 
First, we recall that over $\mathbb{C}$ in the cohomological limit a $q$-difference equation gives rise to a quantum differential equation. Second, we consider a similar construction over $p$-adic numbers $\mathbb{Q}_p$ and then reduce it to the finite field $\mathbb{F}_p$.

\subsection{Quantum connection}
The quantum differential equation for a Nakajima variety $X$ has the form: 
\begin{align} \label{qdecoh}
\nabla_i \tilde{\Psi}_{coh}(z) =0, \qquad \nabla_i =   z_i \frac{\partial}{\partial z_i} - s\, C_i(\zz,\uu), \ \ \ i=1,\dots, l,
\end{align}
where $C_i(\zz,\uu)$ is the operator of quantum multiplication by the first Chern class $c_1(L_i)$ in the equivariant quantum cohomology of $X$, $\uu=(u_1,\dots,u_m)$ denote the equivariant parameters coming from torus action on $X$ and  $s \in \mathbb{C}^{\times}$ denotes the equivariant parameter corresponding to the action of the torus $\mathbb{C}^{\times}$ on the source of the stable maps $C\cong \mathbb{P}^1$. Together, this gives a flat connection $\nabla=(\nabla_1,\dots, \nabla_l)$.

For Nakajima varieties, the three-point functions associated with quantum multiplication by divisors receive no contributions from the orbifold points and are integral-valued, i.e., they lie in $H^{\bullet}_{T}(pt) \cong \mathbb{Z}[\uu]$. Explicitly, Maulik and Okounkov \cite{2012arXiv1211.1287M}  (see Theorem 10.2.1) proved that for a quiver variety $X$ the operator of quantum multiplication by divisor has the following explicit form
\begin{equation} \label{moformula}
   C_i(\zz,\uu)  = c_1(L_i) \cup\, - h \sum\limits_{\alpha \cdot \theta>0} \alpha(c_1(L_i))\frac{\zz^\alpha}{1-\zz^\alpha}e_{\alpha} e_{-\alpha}
\end{equation}
where the sum is taken over the positive roots of the Lie algebra corresponding to the quiver $X$. Positivity is defined by the choice of stability condition $\theta$.  The lowering and raising operators $e_\alpha$ and $e_{-\alpha}$ act in the stable basis of the cohomology of $X$ with integral coefficients. Expanding the rational coefficients in (\ref{moformula}) in the Taylor series in $\zz$, we see that the expansion coefficients are always integers. In summary, the matrix elements of $   C_i(\zz,\uu)$ in the stable basis lie in 
\begin{align} \label{integrality}
(C_i(\zz,\uu))_{a,b} \in \mathbb{Z}[\uu][[\zz]] \cap \mathbb{Q}(\uu,\zz).
\end{align}

Let $\mathbb{Q}_p$ be the field of $p$-adic numbers, $\mathbb{Z}_p\subset \mathbb{Q}_p$ be the ring of integers.
The above integrality allows for base change to $\mathbb{Q}_p$ or to $\mathbb{F}_p$, i.e., a specialization of (\ref{qdecoh}) at $\uu =(u_1,\dots,u_m) \in \mathbb{Z}^m_p$ or at $\uu =(u_1,\dots,u_m) \in \mathbb{F}^m_p$ defines a connection on $\mathbb{Q}^l_p$ and $\mathbb{A}^{l}(\mathbb{F}_p)$ respectively. 

We would like to emphasize the following point. The integrality of quantum multiplication by divisors in the cohomological stable basis allows us to associate to connection $\nabla$ a certain connection over a field of characteristic $p$, which we will denote by the same symbol. The $p$-curvature is an important invariant of this associated field $\mathbb{F}_p$ -- the connection which we investigate below.

\subsection{Quantum connection as limit of $q$-difference equation}
The quantum differential equation (\ref{qdecoh}) can be obtained from the K-theoretic quantum difference equation (\ref{QDEintro}) as follows.
Let $\epsilon$  be a complex parameter with a small complex norm, i.e. $|\epsilon|< 1$. Consider the following substitution:
\begin{align} \label{specpar}
q=1+\epsilon + O(\epsilon^2) ,\quad a_i = q^{s u_i}=  1+s \epsilon u_i + O(\epsilon^2),\quad i=1,\dots,m,
\end{align}
where $s$ is a formal complex parameter, 
i.e., the cohomological equivariant parameters $u_i$ are the first terms in the $\epsilon$-expansions of the K-theoretic equivariant parameters  $a_i$.  Then, the following expansion is known (see \cite{2015CMaPh.334..629B} Theorem 6 and Corollary 7, and \cite{2024arXiv240502473Z} Theorem 7.2):
\begin{equation} \label{Mexp}
{\bf M}_{\cL_i} (\zz,\aa,q)  = 1+ \epsilon \, s\, C_i(\zz,\uu)  +O(\epsilon^2), \ \ \textrm{where} \ \  \uu=(u_1,\dots,u_m),
\end{equation}
This expansion is understood in loc. cit. as expansion of elements in quantum groups acting in the equivariant K-theory and cohomology of $X$. In order to view this as an equality of matrices we need to fix bases in $K_{T}(X)$ and $H_{T}(X)$ which are also related by this limiting procedure. A natural choice is the stable basis of $K_{T}(X)$, which is known to degenerate to the stable basis of $H_{T}(X)$ in the limit $\epsilon \to 0$. The equality (\ref{Mexp}) is then understood as follows: ${\bf M}_{\cL_i} (\zz,\aa,q)$ in the left side of (\ref{Mexp})  is the matrix of the corresponding operator in the K-theoretic stable basis and $C_i(\zz,\uu)$ in the right side is the matrix of quantum multiplication by $c_1(L_i)$ in the stable basis of equivariant cohomology. Note also that for this choice of bases we have the integrality
of the corresponding matrix elements (\ref{integrality}).

Next, let $q^{z_i \frac{\partial}{\partial z_i}}$ denote the operator acting by shifting the K\"ahler parameters $z_i\mapsto z_i q$:
$$
q^{z_i \frac{\partial}{\partial z_i}} f(z_1,\dots, z_i, \dots, z_l) = f(z_1,\dots, z_i q, \dots ,z_l). 
$$
Clearly, from (\ref{specpar}) we have
\begin{align} \label{qdifspec}
q^{z_i \frac{\partial}{\partial z_i}} = 1+ \epsilon z_i \frac{\partial}{\partial z_i} + O(\epsilon^2). 
\end{align}
Let us write the quantum difference equation (\ref{QDEintro}) in the form
\begin{align} \label{kconnec}
\tilde{\Psi}(\zz q^{\cL},\aa,q)  = {\bf M}_{\cL} (\zz,\aa,q) \tilde{\Psi}(\zz,\aa,q), 
\end{align}
where 
$$
\tilde{\Psi}(\zz,\aa,q) = {\Psi}(\zz,\aa,q) e(\zz,\aa,q), \ \ \ e(\zz,\aa,q)= \prod\limits_{i} e^{\frac{\ln(z_i) \ln(L_i)}{\ln(q)}}
$$
so that $e(\zz q^{\cL},\aa,q) = \cL e(\zz,\aa,q)$. This normalization is more useful than (\ref{powerser}) because $\tilde{\Psi}(\zz,\aa,q)$ reduces to $\tilde{\Psi}_{coh}(z)$ from (\ref{qdecoh}) in the cohomological limit $\epsilon \to 0$:
\begin{equation} \label{epsexpans}
\tilde{\Psi}(\zz,\aa,q) = \tilde{\Psi}_{coh}(z) + O(\epsilon)
\end{equation}
We write (\ref{kconnec}) as
$$
\Big({\bf M}_{\cL_i} (\zz,\aa,q)^{-1} q^{z_i \frac{\partial}{\partial z_i}}\Big) \tilde{\Psi}(\zz,\aa,q) =\tilde{\Psi}(\zz,\aa,q) 
$$
Substituting (\ref{Mexp}), (\ref{qdifspec}) and (\ref{epsexpans}) into this equation, we obtain the quantum differential equation (\ref{qdecoh}) as the linear term in the $\epsilon$-expansion.

\subsection{Extension $\mathbb{Q}_p(\pi)$}
Let $p$ be a prime number and let $|\cdot|_p$ denote the multiplicative $p$-adic norm 
$$
|p|_p = \dfrac{1}{p}.
$$
We  consider an extension $\mathbb{Q}_p(\pi)$ where $\pi$ denotes a root of the equation $\pi^{p-1}=-p$. 
Clearly, the $p$-adic norm of $\pi$ equals:
\begin{equation} \label{pinorm}
|\pi|_p = \dfrac{1}{p^{\frac{1}{p-1}}}<1.
\end{equation}
The field $\mathbb{Q}_p(\pi)$ contains all $p$th roots of unity $\zeta_p$, which are of the form
\begin{equation}
\label{proots}
\zeta_p = 1+ b \pi  + O(\pi^2)\,,\qquad b=0,1,\dots, p-1.   
\end{equation} 
In the ring of integers $\mathbb{Z}_p[\pi] \subset \mathbb{Q}_p(\pi)$ the ideal $(\pi)$ is maximal with the residue field
\begin{equation} \label{finfield}
\mathbb{Z}_p[\pi]/(\pi) = \mathbb{F}_p. 
\end{equation} 
Thanks to the relation $\pi^{p-1}=-p$ the $p$-adic expansions in $\pi$ may acquire additional terms which do not appear in the expansions over $\mathbb{C}$ as the following Lemma demonstrates:

\begin{Lem} \label{lem1} 
Let $\alpha$ and $\beta$ be two $N\times N$ matrices with   $|\alpha_{i,j}|_p\leq 1$,  $|\beta_{i,j}|_p\leq 1$. Then 
$$
(1+ \pi \alpha +\pi^2 \beta)^p = 1+\pi^p (\alpha^p-\alpha) + O(\pi^{p+1}) 
$$
\end{Lem}

\begin{proof}
We get
\begin{align}
(1+ \pi \alpha +\pi^2 \beta)^p &=\sum_{k=0}^p \binom{p}{k}(\pi \alpha +\pi^2 \beta)^k
=1 + p(\pi \alpha + \pi^2 \beta) + \frac{p(p-1)}{2}(\pi \alpha + \pi^2 \beta)^2\cr 
&+\dots + (\pi \alpha +\pi^2 \beta)^p\,. \label{psum}
\end{align}
Recall that $p = -\pi^{p-1}$. Since $|\alpha_{i,j}|_p\leq 1$, and  $|\beta_{i,j}|_p\leq 1$ we see that the lowest term $\pi^p$ in the $p$-adic norm appear in the second and the last term of the sum (\ref{psum}): 
\begin{equation}
    (1+ \pi \alpha +\pi^2 \beta)^p = 1 -\pi^p \alpha+\pi^p \alpha^p + O(\pi^{p+1})\,,
\end{equation}
where in the second term $-\pi^p \alpha = p\pi \alpha$.
\end{proof}

\subsection{$p$-curvature}
The $p$-curvature of connection $\nabla_i$ is defined in components by
\begin{equation} \label{pcurvature}
C_p(\nabla_i) =  \nabla_i^p - \nabla_i \pmod{p} 
\end{equation}
Modulo $p$ all derivatives in (\ref{pcurvature}) cancel out and the $p$-curvature is a linear operator.

We assume that the cohomological equivariant parameters have integral $p$-adic values values $\uu=(u_1,\dots,u_m) \in \mathbb{Z}^{m}_p$. Then, from (\ref{integrality}) we obtain
\begin{align} \label{intcoef}
(C_i(\zz,\uu))_{a,b} \in \mathbb{Z}_p[[\zz]] \cap \mathbb{Q}_p(\zz).
\end{align}
Thus reductions modulo $p$ for coefficients of such power are well defined and we obtain
$$
C_p(\nabla_i) \in Mat_{N}(\mathbb{F}_p(z)[s]).
$$ 
An interesting problem is to determine the spectrum of this operator.

\begin{Rem}
The connection (\ref{qdecoh}) is sometimes called `logarithmic connection' in order to distinguish it from the following connection 
\begin{equation}\label{eq:LogConnComm}
    \tilde{\nabla}_i =   \frac{\partial}{\partial z_i} - \frac{s}{z_i}\, C_i(\zz,\uu) , \qquad i=1,\dots, l.
\end{equation}
In terms of $\tilde{\nabla}_i$ the $p$-curvature has a shorter expression due to the following Lemma
\begin{Lem}  \label{lem2}
The following holds modulo $p$
\begin{equation}
\left(\nabla_i\right)^p - \nabla_i = z_i^p\tilde{\nabla}_i^p \pmod{p}
\end{equation}
\end{Lem}

\begin{proof}
First, let us assume that the connection term in \eqref{eq:LogConnComm} is trivial $C_i(\zz,\uu)=0$. Then for a section $f$ (we suppress index $i$ for brevity) we get
$$
(z\tilde{\nabla})^p f(z) = (z\partial)\dots (z\partial)f(z)
$$
where one acts $z\partial=z\frac{\partial}{\partial z}$ on $f$ successively $p$ times. Let
\begin{equation}\label{eq:logconexpn}
    \left(z\partial\right)^n f(z) = z f'(z)+\sum_{k=2}^{n-1}a_{n,k} z^k f^{(k)}(z)+z^n f^{(n)}(z)\,.
\end{equation}
One can immediately observe the following recursive relation between the coefficients $a_{n,k}$
$$
a_{n+1,k} = a_{n,k-1}+k\cdot a_{n,k}\,,
$$
This recursive relation is solved by Stirling numbers of the second kind -- the number of ways to partition a set of $n$ objects into $k$ non-empty subsets:
$$
a_{n,k}=\frac{1}{k!}\sum_{j=1}^k(-1)^{k-j} \binom{k}{j} j^n\,.
$$
When $n=p$ is prime by the little Fermat theorem $j^p\equiv j$ mod $p$ so 
$$
a_{p,k}\equiv\frac{1}{k!}\sum_{j=1}^k(-1)^{k-j} \binom{k}{j} j\,.
$$
The latter expression is equal to zero thanks to the following observation for $k>1$
$$
0=\frac{d}{dx}(1-x)^k\vert_{x=1} = \sum_{j=1}^k(-1)^{j} \binom{k}{j} j\,.
$$
Thus $a_{p,k}\equiv 0 \pmod{p}$ for $1<k<p$ and
\begin{equation}\label{eq:LemmaA0}
\left(z\partial\right)^p f(z) \equiv z f'(z)+z^p\partial^p f(z) \pmod{p}\,.
\end{equation}

The complete statement of the Lemma for the nontrivial connection follows after twisting the equation \eqref{eq:logconexpn}
with the following operator 
\begin{equation}\label{eq:twistingA}
\exp\left[ s\int^z  \frac{C(y,u)}{y} dy\right]\cdot \left(z\partial\right)^n \cdot \exp \left[-s\int^z  \frac{C(y,u)}{y} dy\right] = (z\tilde\nabla)^n
\end{equation}
and similarly on its right hand side before reducing both expressions modulo $p$. All contributions involving coefficient $a_{n,k}$ will remain trivial after the reduction.
\end{proof}
\end{Rem}

\subsection{Expansions at roots of unity}
From (\ref{specpar}) we see that the quantum differential equation appears from the expansion of the $q$-difference equation when $|q-1|<1$ in the complex norm. In this Section, we consider a similar expansion in the $p$-adic norm.  A new feature of this approach is that $q$ may be assumed to be close to a $p$-th root of unity. We show that in the vicinity of these points the quantum difference connections reduce to the corresponding $p$-curvatures.

Assume that $q \in \mathbb{Q}_p(\pi)$ is a $p$-th root of unity. By (\ref{proots}), without loss of generality, we may assume
\begin{equation} \label{exp1}
q =  1+ \pi  + O(\pi^2)\,.
\end{equation}
We also assume that 
\begin{equation}\label{exp2}
a_i = q^{s u_i} = 1+ \pi s u_i + O(\pi^2), \quad i=1,\dots, l. 
\end{equation}
for $u_i \in \mathbb{Z}_p$ and a formal variable $s$ as above. The expansions (\ref{exp1})  and (\ref{exp2}) are the $p$-adic analogs of (\ref{specpar}) where $\pi$ is considered `small' in the $p$-adic norm (\ref{pinorm}).

Using the shift operator, we can write the inverse of the iterated product \eqref{eq:QpoperFull} as
\begin{equation}\label{eq:itprodDef}
{\bf M}_{\cL^{p}_i} (\zz,\aa,q)^{-1} q^{p z_i \frac{\partial}{\partial z_i}} = \left(\mathbf{M}_{\cL_i}(\zz,\aa,q)^{-1} q^{z_i \frac{\partial}{\partial z_i}}\right)^p\,.
\end{equation}
As in (\ref{Mexp}) in the order up to $\pi$ one gets
$$
\mathbf{M}_{\cL_i}(\zz,\aa,q)^{-1} q^{z_i \frac{\partial}{\partial z_i}} =  1+ \pi \nabla_i(z) +O(\pi^2) \,.
$$
Next, thanks to Lemma \ref{lem1} for $\alpha = \nabla_i$ we get
$$
\left(\mathbf{M}_{\cL_i}(\zz,\aa,q)^{-1} q^{z_i \frac{\partial}{\partial z_i}}\right)^p = 1+ \pi^p \left(\nabla_i^p-\nabla_i\right) + O(\pi^{p+1})\,,
$$
or
\begin{equation}
\label{lefteqrr}
\frac{{\bf M}_{\cL^p_i} (\zz,\aa,q)^{-1} -1}{\pi^p} \equiv\left(\nabla_i^p-\nabla_i\right)  \pmod{\pi}.
\end{equation}
Note that by (\ref{finfield}),  this precisely gives the $p$-curvature:
\begin{equation}
\label{lefteq}
\frac{{\bf M}_{\cL^p_i} (\zz,\aa,q)^{-1} -1}{\pi^p} \equiv C_p(\nabla_i)  \pmod{\pi}\,.
\end{equation}

The above analysis demonstrates that (\ref{eq:itprodDef}) considered over $\mathbb{Q}_p(\pi)$ is the correct $q$-difference generalization of the $p$-curvature: it reduces to the $p$-curvature in the first nontrivial term of the $\pi$-expansion around a $p$-th root of unity.

\vspace{3mm}

Next let us consider the same expansion for $\mathbf{M}_{\cL_i}(\zz^p,\aa^p,q^p)^{-1}$.
As in Lemma \ref{lem1} we have
$$
a^p_i=(1+\pi s u_i + O(\pi^2))^p = 1+ \pi^p(s^p u_i^p-s u_i)  +O(\pi^{p+1}).
$$
Since we assume $u_i \in \mathbb{Z}_p$, it follows that $u^p_i = u_i + O(\pi^{p-1})$ (since $u_i^p= u_i \pmod {p}$) we also have
$$
a^p_i=(1+\pi s u_i + O(\pi^2))^p = 1+ \pi^p(s^p-s) u_i^p  +O(\pi^{p+1}).
$$
Using this expansion, from (\ref{Mexp}) we find:
$$
\mathbf{M}_{\cL_i}(\zz^p,\aa^p,q^p)^{-1} = 1-  C_i(\zz^p,\uu^p) \pi^p (s^p-s) + O(\pi^{p+1}).
$$
In other words,
\begin{equation}\label{eqright}
\frac{\mathbf{M}_{\cL_i}(\zz^p,\aa^p,q^p)^{-1}-1}{\pi^{p}} \equiv  (s-s^p) C_i(\zz^p,\uu^p)  \pmod{\pi}\,.
\end{equation}

\subsection{The Isospectrality Theorem} 
Let us summarize the discussion of the previous sections. Let $C_i(\zz,\uu)$ be the operator of quantum multiplication by the divisor $c_1(L_i)$ in the equivariant quantum cohomology of a Nakajima variety. We denote by the same symbol $C_i(\zz,\uu)$ the matrix of this operator in the cohomological stable basis.  Let us specialize the equivariant parameters so that $\uu=(u_1,\dots, u_m) \in \mathbb{Z}^{m}_p$. 
Thanks to (\ref{intcoef}), this matrix has a well defined reduction mod $p$ which gives a  $\mathbb{F}_p$-connection.  The $p$-curvature of this connection is
\begin{equation}
\label{equan1}
C_p(\nabla_i)=(\nabla_i)^p -\nabla_i \in  Mat_{N}(\mathbb{F}_p(\zz)[s]).
\end{equation}

Let $(s-s^p) C(\zz^p,\uu^p)$ be the operator obtained from $ C(\zz,\uu)$ via substitution $\zz^p=(z_1^p,\dots, z_l^p)$ and $\uu^p=(u_1^p,\dots,u_m^p)$ and multiplication by polynomial $s-s^p$. Modulo $p$ we obtain the following matrix
\begin{equation} \label{euqn2}
(s-s^p) C(\zz^p,\uu^p) \in Mat_{N}(\mathbb{F}_p(\zz)[s]).
\end{equation}
\begin{Thm}\label{Th:Isopectral1}
Matrices (\ref{equan1}) and (\ref{euqn2}) have equal sets of the eigenvalues.  
\end{Thm}
\begin{proof}
The operator in (\ref{equan1}) corresponds to the right-hand side of (\ref{lefteq}), while the operator in (\ref{euqn2}) is defined by the right-hand side of (\ref{eqright}). Congurences modulo $\pi$ imply congruences modulo $p$ since $\pi^{p-1}=-p$. Finally, by Theorem \ref{Th:Isopectral}, the left-hand sides of (\ref{lefteq}) and (\ref{eqright}) share the same spectrum.
\end{proof}

For a matrix $A=(a_{i,j})$ let $A^{(1)}$ denote its Frobenius twist, i.e., the matrix obtained from $A$ by raising all matrix elements to the $p$-th power $A^{(1)}=(a^p_{i,j})$. Clearly, over a field of characteristic $p$, we have
$$
C(\zz,\uu)^{(1)}  =  C(\zz^p,\uu^p)
$$
We then can reformulate the last theorem in the form in which it was formulated in \cite{EV2024}:

\begin{Thm}[\cite{EV2024}]\label{Th:EVpencil}
    The spectra of the periodic pencil $(s-s^p)C(\zz,\uu)^{(1)}$ and the $p$-curvature $C_p(\nabla_i)$ are isomorphic over field of characteristic $p$.
\end{Thm}

Finally, we note that the spectrum of the quantum operators $C_i(\zz,\uu)$ has an explicit description in terms of Bethe Ansatz \cite{Aganagic:2017be}. The last theorem thus fully determines the spectrum of the $p$-curvature.

\bibliography{cpn1}

\end{document}